\documentclass[12pt]{elsarticle}
\usepackage{fullpage}
\usepackage[cp1251]{inputenc}
\usepackage[english]{babel}
\usepackage{amsmath}
\usepackage{amssymb}
\usepackage{amsfonts}
\usepackage{mathtools}
\usepackage{dsfont}
\usepackage{rotating}
\usepackage{graphicx}
\usepackage{floatflt,epsfig}
\usepackage{lineno,hyperref}
\usepackage{enumerate}
\usepackage{comment}
\usepackage{colortbl}
\usepackage{array,tabularx,tabulary,booktabs}
\usepackage{longtable}
\usepackage{multirow}
\usepackage{wrapfig}
\usepackage{subcaption}
\usepackage[table]{xcolor}
\newcolumntype{^}{>{\currentrowstyle}}

\journal{Arxiv}
\newtheorem{theorem}{Theorem}[section]
\newtheorem{lemma}[theorem]{Lemma}

\newtheorem{corollary}[theorem]{Corollary}
\newtheorem{proposition}[theorem]{Proposition}

\newtheorem{conjecture}[theorem]{Conjecture}
\newtheorem{remark}[theorem]{Remark}
\newtheorem{example}[theorem]{Example}

\newcommand{\F}{\mathbb{F}}

\begin{document}
\renewcommand{\abstractname}{Abstract}
\renewcommand{\refname}{References}
\renewcommand{\tablename}{Figure.}
\renewcommand{\arraystretch}{0.9}
\thispagestyle{empty}
\sloppy

\begin{frontmatter}
\title{On eigenfunctions and maximal cliques of generalised Paley graphs of square order}

\author[01]{Sergey Goryainov}
\ead{sergey.goryainov3@gmail.com}

\author[02]{Leonid Shalaginov}
\ead{44sh@mail.ru}

\author[03]{Chi Hoi Yip}
\ead{kyleyip@math.ubc.ca}
 
\address[01] {School of Mathematical Sciences, Hebei Workstation for Foreign Academicians, \\Hebei Normal University, Shijiazhuang  050024, P.R. China}
\address[02] {Chelyabinsk State University, 129 Bratiev Kashirinykh st.,\\Chelyabinsk 454021, Russia}
\address[03] {Department of Mathematics, University of British Columbia,\\ Vancouver V6T 1Z2, Canada}

\begin{abstract}
Let GP$(q^2,m)$ be the $m$-Paley graph defined on the finite field with order $q^2$. We study eigenfunctions and maximal cliques in generalised Paley graphs GP$(q^2,m)$, where $m \mid (q+1)$. In particular, we explicitly construct maximal cliques of size $\frac{q+1}{m}$ or $\frac{q+1}{m}+1$ in GP$(q^2,m)$, and 
show the weight-distribution bound on the cardinality of the support of an eigenfunction is tight for the smallest eigenvalue $-\frac{q+1}{m}$ of GP$(q^2,m)$. These new results extend the work of Baker et. al and Goryainov et al. on Paley graphs of square order. We also study the stability of the Erd\H{o}s-Ko-Rado theorem for GP$(q^2,m)$ (first proved by Sziklai).
\end{abstract}

\begin{keyword}
generalised Paley graph; maximal clique; eigenfunction; affine plane; orthogonal array;  Erd\H{o}s-Ko-Rado theorem
\vspace{\baselineskip}
\MSC[2010] 05C25\sep 51E15 (11T30\sep 05E30\sep 30C10 \sep 05C69)
\end{keyword}
\end{frontmatter}
%%%%%%%%%%%%%%%%%%%%%%%%%%%%%%%%%%%%%%%%%%%%%%%%%%%%%%%%%%%%%%%%%%%%%%%%%%%%%%%%%%%%%%%%%%%%%%%%%%%%%%%%%%%%%%%%%%%%%%%%%%%%%%%%%%

\section{Introduction}

Throughout the paper, let $p$ be an odd prime and $q$ a power of $p$. Let $\F_q$ be the finite field with $q$ elements, $\F_q^+$ be its additive group, and $\F_q^*=\F_q \setminus \{0\}$ be its multiplicative group. 

Generalised Paley graphs are natural generalisations of the well-studied Paley graphs. They were first introduced by
Cohen \cite{C88} in 1988, and reintroduced by Lim and Praeger \cite{LP09} in 2009. Let $m > 1$ be a positive integer and let $q$ be an odd prime power such that $q \equiv 1 \pmod {2m}$. The {\em $m$-Paley graph} on $\mathbb{F}_q$, denoted $\mathrm{GP}(q,m)$, is the Cayley graph $\operatorname{Cay}(\mathbb{F}_q^+,(\mathbb{F}_q^*)^m)$, where
$(\mathbb{F}_q^*)^m$ is the set of $m$-th powers in $\mathbb{F}_q^*$. More explicitly, the vertex set of $\mathrm{GP}(q,m)$ is given by $\F_q$, and two vertices are adjacent if and only if their difference is an $m$-th power in $\mathbb{F}_q^*$. Note that the condition $q \equiv 1 \pmod {2m}$ guarantees that $\mathrm{GP}(q,m)$ is undirected, so we can talk about its cliques. Note that $2$-Paley graphs are precisely Paley graphs and $3$-Paley graphs are often referred to as cubic Paley graphs. Unsurprisingly, it often requires non-trivial efforts to extend results on Paley graphs to generalised Paley graphs, partly because Paley graphs are known to have the highest degree of symmetry \cite{P01}.

The problem of finding the clique number of a generalised Paley graph remains widely open in general; see \cite{HP21, Yip21, Y21, Yip22} for recent (minor) improvements on the upper bound on the clique number. It is also interesting to classify maximum cliques and maximal cliques in generalised Paley graphs. Recall that a clique in a graph is called \emph{maximum} if there is no clique of larger size in this graph, while a clique in a graph is called \emph{maximal} if it is not included in a clique of larger size. The clique number of a graph $X$,  denoted by $\omega(X)$, is the size of a maximum clique.

Throughout the paper, we consider generalised Paley graphs $\mathrm{GP}(q^2,m)$, where $m \mid (q+1)$ and $m>1$. These graphs belong to the family of semi-primitive cyclotomic strongly regular graphs (see \cite[Section 4]{Xiang12}) and have attracted the attention of many researchers. Given $m \mid (q+1)$ and $m>1$, the clique number of the generalised Paley graph $\mathrm{GP}(q^2,m)$ is known to be $q$; moreover, we have the following simple characterization of its maximum cliques. 

\begin{lemma}[{\cite{B84}, \cite[Theorem 1.2]{S99}}]\label{TSziklai}
Let $q$ be an odd prime power and $m$ a divisor of $(q+1)$ such that $m > 1$, then in the generalised Paley graph $\mathrm{GP}(q^2,m)$, the only maximum clique containing $0,1$ is the subfield $\mathbb{F}_q$; in particular, each maximum clique in $\mathrm{GP}(q^2,m)$ is the image of the subfield $\F_q$ under an affine transformation.
\end{lemma}

Lemma~\ref{TSziklai} was first proved by Blokhuis \cite{B84} on Paley graphs of square order (i.e., when $m=2$), and was later extended by Sziklai for all $m \geq 3$ which divides $(q+1)$. We remark that the assumption $m | (q + 1)$ in Lemma~\ref{TSziklai} is crucial. In the case $m \nmid (q + 1)$, it is easy to verify that the subfield $\mathbb{F}_q$ does not form a clique in $\mathrm{GP}(q^2, m)$ \cite[Lemma 2.2]{Y21};
moreover, Yip \cite{Y21} recently proved that $\omega(\mathrm{GP}(q^2, m)) \le q - 1$. Therefore, it is
unlikely that there is a nice characterization of maximum cliques in $\mathrm{GP}(q^2, m)$, where $m \nmid (q + 1)$; in fact, it is not even known what should be the right order of magnitude of $\omega(\mathrm{GP}(q^2, m))$.

Lemma~\ref{TSziklai} can be regarded as the Erd\H{o}s-Ko-Rado (EKR) theorem for the generalised Paley graph $\mathrm{GP}(q^2,m)$. If we view cliques which are images of an affine transformation of the subfield $\F_q$ as {\em canonical cliques}, then Lemma~\ref{TSziklai} states that each maximum clique is canonical; see the related discussion in \cite{AGLY22} and\cite[Section 5.9]{GM15}. Sziklai's proof of Lemma~\ref{TSziklai} is a nice application of the polynomial method, and a new simple proof is given in \cite{AY21} using character sums and finite geometry. Both these proofs implicitly viewed $\mathrm{GP}(q^2,m)$ geometrically. In particular, an explicit isomorphism between $\mathrm{GP}(q^2,m)$ and block graphs of orthogonal array $OA(\frac{q+1}{m},q)$ coming from the point-line incidence structure of the affine Galois plane $AG(2,q)$ is discussed in \cite{AGLY22}. In the case $m>(q+1)/(\sqrt{q}+1)$, a simple self-contained proof of Lemma~\ref{TSziklai} is given in \cite[Corollary 8]{AGLY22}. We will borrow ideas from these three different viewpoints to study maximal cliques. 

The next step to understand the structure of large cliques in $\mathrm{GP}(q^2,m)$ is to study the {\em stability} of canonical cliques: if a clique has a size close to the clique number, is it guaranteed that such a clique must be contained in a canonical clique? Such a result is important in extremal combinatorics (in particular Hilton-Milner type results \cite{HM67}) and finite geometry \cite[Page 548]{S99}. The next lemma describes the best-known result; see also \cite[Corollary 4.3]{AY21}. 

\begin{lemma}[{\cite[Theorem 1.3]{S99}}] \label{embed}
Let $q$ be an odd prime power and $m$ a divisor of $(q+1)$ such that $m \geq 3$. If $C$ is a clique in the generalised Paley graph $\mathrm{GP}(q^2,m)$, such that $0,1 \in C$, and $|C|>q-(1-1/m)\sqrt{q}$, then $C \subset \F_q$. 
\end{lemma}

Lemma~\ref{embed} is a consequence of the stability of the direction set determined by point sets in the affine plane $AG(2,q)$ (see \cite[Theorem 4]{Sz96} and \cite[Theorem 3.1]{S99}). It is conjectured that the lower bound on $|C|$ can be improved significantly. In particular, when $m=2$, it is conjectured in \cite{BEHW96, GKSV18, GMS22} that there is no maximal clique with size between $\frac{q+3}{2}$ and $q$ (exclusive). Equivalently, in the case $m=2$, it is conjectured that the lower bound on $|C|$ can be improved to $\frac{q+3}{2}$; however, even the weaker bound $q-\sqrt{q}/2$ stated in Lemma~\ref{embed} is not known to be true. For $m$-Paley graphs, it is reasonable to conjecture that the lower bound can be improved to $q/m+O(1)$. While this is out of reach and is not always true (see Example~\ref{counterexample} for counterexamples), it is still believable that the lower bound $q-O(\sqrt{q})$ could be improved significantly. We provide a significant improvement on Lemma~\ref{embed} in  Corollary~\ref{corstable}, provided that $m>\sqrt{q}$.

In \cite{BEHW96} and \cite{GKSV18}, maximal but not maximum cliques of Paley graphs of square order were constructed. A reinterpretation of these constructions was given in \cite{GMS22}. In \cite{GKSV18}, it is also shown that the weight-distribution bound (see Section~\ref{wdb}) on the minimum cardinality of the support of an eigenfunction is tight for non-principal eigenvalues of Paley graphs of square order. It is known that the tightness of the weight-distribution bound is closely to related to structures of induced subgraphs; we refer to \cite{KMP16,SV21} for a general discussion. Motivated by these results, in this paper, we construct new families of cliques in $\mathrm{GP}(q^2,m)$, and generalise the above results on Paley graphs of square order to $\mathrm{GP}(q^2,m)$. 

Inspired by Lemma \ref{TSziklai} and Lemma~\ref{embed}, maximal cliques with subfield or subspace structure in Cayley graphs over fields are considered in \cite{Y22}. For example, an explicit construction of maximal but not maximum cliques of size $q^2$ in cubic Paley graphs of order $q^6$ was given: the construction is simply the subfield with $q^2$ elements \cite[Theorem 1.5]{Y22}. Interestingly, we shall see in Example~\ref{counterexample} that this special family of cliques shows that our main results cannot be improved significantly since it provides an obstruction in extending the previous works on Paley graphs to generalised Paley graphs.

Our main results are Theorem~\ref{thm: WDBtight}, Theorem~\ref{thm: nearlymaximal}, and Theorem~\ref{thm: stable}. We will combine different ideas from algebraic graph theory, finite geometry, group theory, and number theory.

The paper is organized as follows. In Section~\ref{prelim}, we recall some basic terminologies and give some preliminary results. In Section~\ref{sec: geometry}, we study properties of generalised Paley graphs from the viewpoint of finite geometry. In Theorem~\ref{thm: WDBtight} of Section~\ref{sec: construction}, we show the weight-distribution bound is tight for the smallest eigenvalue $-\frac{q+1}{m}$ of the generalised Paley graph $\mathrm{GP}(q^2,m)$. We also construct new families of cliques in generalised Paley graphs and discuss their correspondences in Section~\ref{sec: construction}. We will discuss the maximality of these cliques in Section~\ref{partialresult} and prove Theorem~\ref{thm: nearlymaximal}. Finally, in Section~\ref{stability}, we study the stability of canonical cliques and improve Lemma~\ref{embed} under extra conditions in Theorem~\ref{thm: stable}.

\section{Preliminaries}\label{prelim}
In this section we list some preliminary definitions and results.
\subsection{The weight-distribution bound for strongly regular graphs} \label{wdb}
A $k$-regular graph on $v$ vertices is called \emph{strongly regular} with parameters $(v,k,\lambda,\mu)$ if any two adjacent vertices have $\lambda$ common neighbours and any two distinct non-adjacent vertices have $\mu$ common neighbours. 
If $X$ is a strongly regular graph, then its complement is also a strongly regular
graph. A strongly regular graph $X$ is \emph{primitive} if both $X$ and its complement
are connected. If $X$ is not primitive, we call it \emph{imprimitive}. The imprimitive
strongly regular graphs are exactly the disjoint unions of complete graphs and
their complements, namely, the complete multipartite graph. We focus on primitive strongly regular graphs.

\begin{lemma}[{\cite[Theorem 5.2.1]{GM15}}]\label{EigenvaluesSRG}
If $X$ is a primitive strongly regular graph with parameters $(v,k,\lambda,\mu)$ and  
$$
\Delta:=\sqrt{(\lambda-\mu)^2+4(k-\mu)},
$$
then $X$ has exactly three eigenvalues 
$$
k,~~~\theta_1 = \frac{\lambda-\mu+\Delta}{2},~~~\theta_2 = \frac{\lambda-\mu-\Delta}{2},
$$
with respective multiplicities
$$
m_k = 1,~~~m_{\theta_1} = -\frac{(v-1)\theta_2+k}{\theta_1-\theta_2},~~~
m_{\theta_2} = \frac{(v-1)\theta_1+k}{\theta_1-\theta_2}.
$$
\end{lemma}
Note that $\theta_2 < 0 < \theta_1$ holds. For the primitive strongly regular graph $X$ from the above lemma, the matrix
$$
\left(
  \begin{array}{c|cc}
    1 & k & v-1-k\\
    m_{\theta_1} & \theta_1 & -1-\theta_1 \\
    m_{\theta_2} & \theta_2 & -1-\theta_2\\
  \end{array}
\right).
$$
is called the \emph{modified matrix of eigenvalues}. The first column gives the dimensions of the eigenspaces (i.e., the multiplicities of the eigenvalues); the second column contains the eigenvalues of the graph, and the third gives the eigenvalues of its complement.

Let $\theta$ be an eigenvalue of a graph $X$. A real-valued function $f$ on the vertex set of $X$ is called an \emph{eigenfunction}
of the graph $X$ corresponding to the eigenvalue $\theta$ (or a \emph{$\theta$-eigenfunction} of $X$), if $f \not \equiv 0$ and
for any vertex $\gamma$ in $X$ the condition
\begin{equation}\label{LocalCondition}
\theta\cdot f(\gamma)=\sum_{\substack{\delta\in{X(\gamma)}}}f(\delta)
\end{equation}
holds, where $X(\gamma)$ is the set of neighbours of the vertex $\gamma$.
Although eigenfunctions of graphs receive less attention of researchers in contrast to their eigenvalues, there are still tons of related literature. We refer to the recent survey \cite{SV21} for a summary of results on the problem of finding the minimum cardinality of support of eigenfunctions of graphs and characterising the optimal eigenfunctions.

The following lemma gives a lower bound for the number of non-zeroes (i.e., the cardinality of the support) for an eigenfunction of a strongly regular graph. In fact, this is a special case of a more general result for distance-regular graphs \cite[Section 2.4]{KMP16}. 
\begin{lemma}\label{WDB}
Let $X$ be a primitive strongly regular graph with parameters $(v,k,\lambda,\mu)$ and let $\theta$ be a non-principal eigenvalue of $X$. Then an eigenfunction of $X$ corresponding to the eigenvalue $\theta$ has at least $$1+|\theta|+\bigg|\frac{(\theta-\lambda)\theta-k}{\mu}\bigg|$$ non-zeroes.
\end{lemma}
\begin{proof}
If follows from \cite[Corollary 1]{KMP16} and the fact that a primitive strongly regular graph with parameters $(v,k,\lambda,\mu)$ is a distance-regular graph of diameter $2$ with $a_1 = \lambda$, $b_0 = k$ and $c_2 = \mu$. $\square$
\end{proof}

\medskip

The lower bound given in Lemma~\ref{WDB} is known as the {\em weight-distribution bound}. Next we explicitly determine the weight-distribution bound in terms of the corresponding eigenvalue.

\begin{corollary}\label{WDBsrg}
Let $X$ be a primitive strongly regular graph with non-principal eigenvalues $s$ and $r$, such that $s<0<r$. Then an eigenfunction of $X$ corresponding to the eigenvalue $r$ has at least $2(r+1)$ non-zeroes, and an eigenfunction corresponding to the eigenvalue $s$ has at least $-2s$ non-zeroes.
\end{corollary}

\begin{proof}
Suppose $X$ is strongly regular with parameters $(v,k,\lambda,\mu)$. Note that $s,r$ are the two solution of the quadratic equation $x^2-(\lambda-\mu)x-(k-\mu)=0$ according to Lemma~\ref{EigenvaluesSRG}. Thus, $(s-\lambda)s-k=-\mu(s+1)$ and $(r-\lambda)r-k=-\mu(r+1)$. It follows that
$$
1+|s|+\bigg|\frac{(s-\lambda)s-k}{\mu}\bigg|=1-s+\bigg|\frac{-\mu(s+1)}{\mu}\bigg|=1-s-(s+1)=-2s,
$$
$$
1+|r|+\bigg|\frac{(r-\lambda)r-k}{\mu}\bigg|=1+r+\bigg|\frac{-\mu(r+1)}{\mu}\bigg|=1+r+(r+1)=2(r+1).
$$
The proof finishes by applying Lemma~\ref{WDB}. $\square$
\end{proof}

\medskip

We will discuss the tightness of the weight-distribution bound of generalised Paley graphs in Section~\ref{TightnessWDB}, and see its combinatorial interpretation in terms of the existence of an induced complete bipartite subgraph of a particular size (Remark~\ref{connection}). 

\subsection{Block graphs of orthogonal arrays} \label{BGOA}
In this section we recall basic facts on the block graphs of orthogonal arrays; for more details see \cite[Section 5.5]{GM15}.

An \emph{orthogonal array} $OA(w, n)$ is an $w \times n^2$ array with entries from an $n$-element set $A$
with the property that the columns of any $2 \times n^2$ subarray consist of all $n^2$ possible pairs.
The \emph{block graph of an orthogonal array} $OA(w, n)$ is defined to be the graph
whose vertices are columns of the orthogonal array, where two columns are
adjacent if there exists a row where they have the same entry. We denote the
block graph for an orthogonal array $OA(w, n)$ by $X_{OA(w,n)}$. It is known that $X_{OA(w,n)}$ is strongly regular.

\begin{lemma}[{\cite[Theorem 5.5.1]{GM15}}]\label{OAGraph}
If $OA(w, n)$ is an orthogonal array where $w < n+1$, then
its block graph $X_{OA(w,n)}$ is strongly regular with parameters
$$
(n^2, w(n - 1), (w - 1)(w - 2) + n - 2, w(w - 1)).
$$
and modified matrix of eigenvalues
$$
\left(
  \begin{array}{c|cc}
    1 & w(n-1) & (n-1)(n+1-w)\\
    w(n-1) & n-w & w-n-1 \\
    (n-1)(n+1-w) & -w & w-1\\
  \end{array}
\right).
$$
\end{lemma}

Next we compute the weight-distribution bound of the block graph of an orthogonal array.

\begin{lemma}\label{WDBOA}
If $OA(w, n)$ is an orthogonal array where $w < n+1$, then
any $(-w)$-eigenfunction of the block graph $X_{OA(w,n)}$ has at least $2w$ non-zeroes, and any $(n-w)$-eigenfunction of the block graph $X_{OA(w,n)}$ has at least $2(n+1-w)$ non-zeroes.
\end{lemma}
\begin{proof}
The lemma follows from Lemma~\ref{OAGraph} and Corollary~\ref{WDBsrg}. $\square$
\end{proof}
\medskip

Let $S_{r,i}$ be the set of columns of $OA(w, n)$ that have the entry $i$ in row $r$. These cliques
are called the \emph{canonical cliques} in the block graph $X_{OA(w,n)}$. Note that these cliques have size $n$. It is known that when $n>(w-1)^2$, these canonical cliques are the only maximum cliques in the block graph $X_{OA(w,n)}$; in other words, in this case $X_{OA(w,n)}$ has the strict-EKR property \cite[Theorem 17]{AGLY22}. 

The following lemma is the key to proving the above strict-EKR property of block graphs of orthogonal arrays. We include a short proof for the sake of completeness, especially because we will refer to the proof of this lemma in later discussions.

\begin{lemma}[{\cite[Theorem 5.5.2]{GM15}}]\label{OAmaximal}
Let $X=X_{OA(w,n)}$ be the block graph of an orthogonal array $OA(w, n)$. If $w \leq 2$, then each maximal clique in $X$ is maximum. If $w \geq 3$, and $C$ is a maximal clique in $X$ and $C$ is not a canonical clique, then $|C| \leq (w-1)^2$. 
\end{lemma}

\begin{proof}
If $w=1$, the orthogonal array only has one row, and it is clear that each maximal clique is maximum. Next we assume $w \geq 2$.

By possibly relabelling the entries in $OA(w,n)$, we may assume that the entries are from $\{0, 1, \ldots, n-1\}$. Let $C$ be a maximal non-canonical clique in $X$. By translating entries of each row modulo $n$, we may assume, without loss of generality, that the column $c_0$ with all zeros is in $C$. Then other columns in $C$ must each have exactly one zero entry. Let $D_i$ be the set of columns in $C\setminus \{c_0\}$ having a zero entry in the $i$-th row. Without loss of generality, we assume that $|D_1| \leq |D_2| \leq \cdots \leq |D_w|$. It follows that 
\begin{equation}\label{ubmaximal}
|C| = 1+\sum_{j=1}^{w}|D_j| \leq 1+w|D_w|.
\end{equation}
Since $C$ is a maximal non-canonical clique, it is not contained in any canonical clique, thus we have $0<|D_{w-1}|\leq |D_w|$. Let $c$ be a column in $D_{w-1}$; then for each $1 \leq j \leq w-2$, there is at most one column in $D_w$ that shares the same entry with $c$ in row $j$. And if $j \in \{w-1,w\}$, there is no column in $D_w$ that shares the same entry with $c$ in row $j$. Therefore, $|D_w| \leq w-2$ and  $|C| \leq 1+w|D_w| \leq (w-1)^2.$ 
Note that when $w=2$, this means $|C|=1$, which implies that there is no maximal non-canonical clique when $w=2$. $\square$
\end{proof}

\medskip

We remark that it is possible to construct an orthogonal array $OA(w,n)$ for which the upper bound $(w-1)^2$ in the above lemma is tight. See the related discussion in Proposition~\ref{prop:upperbound} and Example~\ref{counterexample}; see also \cite[Section 5.5]{GM15}.

\subsection{Affine plane $AG(2,q)$} 

Denote by $AG(2,q)$ the point-line incidence structure, whose points are the vectors of the $2$-dimensional vector space $V(2,q)$ over $\mathbb{F}_q$, and the lines are the additive shifts of $1$-dimensional subspaces of $V(2,q)$. It is well-known that $AG(2,q)$ satisfies the axioms of a finite affine plane of order $q$.
In particular, each line contains $q$ points, and there exist $q+1$ lines through a point.

An \emph{oval} in the affine plane $AG(2,q)$ is a set of $q+1$ points such that no three are on a line.
A line meeting an oval in one point (resp. in two points) is called \emph{tangent} (resp. \emph{secant}).
For any point of an oval, there exists a unique tangent at this point and $q$ secants.

\subsection{Automorphisms of $\mathrm{GP}(q^2,m)$} 
The automorphism group of $\mathrm{GP}(q^2,m)$ was studied in \cite{LP09}. 

\begin{lemma}[{\cite[Section 2.2]{LP09}}]\label{AutGP}
    The automorphism group of $\mathrm{GP}(q^2,m)$ contains a subgroup acting arc-transitively, which is given by
    $$\big\{\, \gamma \mapsto a \gamma^\varepsilon+b~\big|~a \in S,~b \in \mathbb{F}_{q^2},~\varepsilon \in {\rm Gal}(\mathbb{F}_{q^2}) \,\big\},$$
    where $S = \{1,\beta^{m},\beta^{2m},\ldots,\beta^{q^2-1-m}\}$ is the set of $m$-th powers of a primitive element of $\mathbb{F}_{q^2}$.
    \end{lemma}

    \begin{lemma}\label{AutGPSubgroup} Let $q$ be an odd prime power and $m \ge 2$ be an integer such that $m$ divides $q+1$. The subgroup $\{\gamma \mapsto a\gamma+b~|~a \in \mathbb{F}_q^*,b \in \mathbb{F}_q\}$ in ${\rm Aut} \left(\mathrm{GP}(q^2,m) \right)$ stabilises the line $\mathbb{F}_q$ and acts faithfully on the set of points that do not belong to $\mathbb{F}_q$. 
    \end{lemma}
\begin{proof} 
It is easy to see that the subgroup stabilises the line $\mathbb{F}_q$. Then standard computations show that, for any distinct $\gamma_1, \gamma_2$ that do not belong to $\mathbb{F}_q$, there exists a unique automorphism from the subgroup sending $\gamma_1$ to $\gamma_2$.
$\square$
\end{proof}

\subsection{Character sums over affine lines}
A character $\chi$ of the multiplicative group $\F_q^*$ of $\F_q$ is called a {\em multiplicative character} of $\F_q$. For a multiplicative character $\chi$, its order $d$ is the smallest positive integer such that $\chi^d=\chi_0$, where $\chi_0$ is the trivial multiplicative character of $\F_q$.

Weil's bound is useful in estimating a complete character sum over a finite field. However, for many applications, we need to deal with incomplete character sums. The following lemma, due to Katz \cite{Katz89}, provides a Weil-type estimate for character sums over affine lines. Recall that an element $\theta \in \F_{q^n}$ is said to have {\em degree} $k$ over $\F_q$ if $\F_{q^k}$ is the smallest extension of $\F_q$ that contains $\theta$; note that the degree is necessarily a divisor of $n$.

\begin{lemma}[\cite{Katz89}]\label{Katz} 
Let $\theta$ be an element of degree $n$ over $\mathbb{F}_{q}$ and $\chi$ a non-trivial multiplicative character of $\mathbb{F}_{q^{n}} .$ Then
$$
\left|\sum_{a \in \mathbb{F}_{q}} \chi(\theta+a)\right| \leq(n-1) \sqrt{q}.
$$
\end{lemma}

The following corollary is a consequence of Lemma~\ref{Katz}. It says that there cannot be too many consecutive $m$-th powers in a generic line.

\begin{corollary} \label{cor: consecutivempowers} 
Let $q=p^n$ and $m \mid (q+1)$ with $m>1$. Let $u \in \F_{q^2}$, such that $u$ is of degree $2n$ over $\F_p$. If $p>(2n-1)^2$, then there is an element in $u+\F_p$ which is not an $m$-th power in $\F_{q^2}$.
\end{corollary}
\begin{proof}
Let $\chi$ be a multiplicative character of $\F_{q^2}$ with order $m$. Suppose that each element in $u+\F_p$ is an $m$-th power in $\F_{q^2}$. Then we have $\sum_{a \in \F_p} \chi(u+a)=p$ since $0 \notin u+\F_p$. On the other hand, since $u$ is of degree $2n$ over $\F_p$, Lemma~\ref{Katz} implies that 
$$
p=\left|\sum_{a \in \F_p} \chi(u+a)\right| \leq(2n-1) \sqrt{p}.
$$
Thus, we must have $p \leq (2n-1)^2$, violating our assumption. Thus, there is an element in $u+\F_p$ which is not an $m$-th power in $\F_{q^2}$. $\square$
\end{proof}

\section{Geometric properties of generalised Paley graphs of square order}\label{sec: geometry}

In this section we view the generalised Paley graph $\mathrm{GP}(q^2,m)$ geometrically and explore its properties. As always, we assume that $m>1$ and $m \mid (q+1)$.

\subsection{Finite field of order $q^2$ as the affine plane $AG(2,q)$} \label{FF}

Let $d$ be a non-square in $\mathbb{F}_{q}^*$. The set of elements of a finite field of order $q^2$ can be considered as $\mathbb{F}_{q^2} =
\{x+y\alpha~|~x,y \in \mathbb{F}_q\}$, where $\alpha$ is a root of the polynomial $f(t) = t^2 - d$.
Since $\mathbb{F}_{q^2}$ is a $2$-dimensional vector space over $\mathbb{F}_q$, we can identity the points of $AG(2,q)$ as the elements of $\mathbb{F}_{q^2}$, and identity a line $l$ in $AG(2,q)$ as $\{x_1+y_1\alpha + c(x_2+y_2\alpha) ~|~c \in \mathbb{F}_q \}$, where $x_1+y_1\alpha \in \mathbb{F}_{q^2}$, $x_2+y_2\alpha \in \mathbb{F}_{q^2}^*$ are fixed and $c$ runs over $\mathbb{F}_q$. 
The element $x_2+y_2\alpha$ is called the \emph{slope} of the line $l$.
Note that the slope is defined up to multiplication by a scalar in $\mathbb{F}_{q}^*$.
A line $l$ in $\mathbb{F}_{q^2}$ is called \emph{quadratic}, \emph{cubic} or, more generally, \emph{$m$-ary} if its slope is a square, a cube or an $m$-th power in $\mathbb{F}_{q^2}^*$, respectively.

Let $\beta$ be a primitive element of the finite field $\mathbb{F}_{q^2}.$
Since $m \mid (q+1)$, the elements of $\mathbb{F}_{q}^* = \langle\beta^{q+1}\rangle$ are $m$-th powers in $\mathbb{F}_{q^2}^*$, and thus the difference between any two points of an $m$-ary line is an $m$-th power in $\mathbb{F}_{q^2}^*$.
In this setting, two distinct vertices are adjacent in $\mathrm{GP}(q^2, m)$ if and only if the line through these points is $m$-ary. 
\begin{lemma}\label{lem: numberoflines}
There are exactly $\frac{q+1}{m}$ many $m$-ary lines through any point.
\end{lemma}
\begin{proof}
If suffices to prove the statement for the point $0$. Note that the number of $m$-th powers in $\mathbb{F}_{q^2}^*$ is $\frac{q^2-1}{m}$. Since $m$ divides $q+1$ and $\mathbb{F}_q^* = \langle\beta^{q+1}\rangle$ for a primitive element from $\mathbb{F}_{q^2}$, every element from $\mathbb{F}_{q}^*$ is an $m$-th power in $\mathbb{F}_{q^2}^*$. It implies that every $m$-ary line through $0$ (such a line can be represented as $\{c\gamma~|~c \in \mathbb{F}_q\}$ for some $m$-th power $\gamma \in \mathbb{F}_{q^2}^*$) has exactly $q-1$ (all points except $0$) $m$-th powers from $\mathbb{F}_{q^2}^*$. Thus, the number of $m$-ary lines through $0$ is $\frac{q^2-1}{m}\cdot\frac{1}{q-1}$, which is equal to $\frac{q+1}{m}$. $\square$  
\end{proof}

\medskip

There are two special lines in $AG(2,q)$, for which we can decide in general if they are $m$-ary: the line $\mathbb{F}_q = \langle\beta^{q+1}\rangle$, which is $m$-ary since $m$ divides $q+1$, and the line $\{c\alpha~|~c \in \mathbb{F}_{q}\}$, for which the following lemma gives the criterion.

\begin{lemma}\label{calpha}
The line $\{c\alpha~|~c \in \mathbb{F}_q\}$ is $m$-ary if and only if $m$ divides $\frac{q+1}{2}$.
\end{lemma}
\begin{proof}
It is clear that $\beta^{\frac{q+1}{2}}$ belongs to the line $\{c\alpha~|~c\in \mathbb{F}_q\}$ and thus it is a slope of $\{c\alpha~|~c\in \mathbb{F}_q\}$. This slope is an $m$-th power in $\mathbb{F}_{q^2}^*$ if and only if $m$ divides $\frac{q+1}{2}$.
$\square$
\end{proof}

\subsection{Generalised Paley graphs of square order as block graphs of orthogonal arrays} \label{construction}

In \cite[Section 3]{AGLY22}, it is shown that each Peisert-type graph (see \cite[Definition 2]{AGLY22}) is isomorphic to the block graph of an orthogonal array. In particular, generalised Paley graph $\mathrm{GP}(q^2,m)$ with $m \mid (q+1)$ is isomorphic to the block graph of an orthogonal array $OA(\frac{q+1}{m},q)$. In this section, we present a slightly different realization of this fact, which is more convenient for our discussion.

Given an odd prime power $q$, recall the construction of an orthogonal array $OA(q+1,q)$ based on the affine plane $AG(2,q)$.
For the affine plane $AG(2,q)$ defined by the finite field $\mathbb{F}_{q^2}$, all the $q+1$ slopes are given by the elements $\{1\} \cup \{c+\alpha~|~c \in \mathbb{F}_q\}$. These $q+1$ slopes define an orthogonal array $OA(q+1,q)$ as follows. The rows are indexed by the slopes (by the lines through 0). The columns are indexed by the elements of $\mathbb{F}_{q^2}$. In order to define the orthogonal array, we need to assign an index to every line in the affine plane (the indexes are from $\mathbb{F}_q$). Given a line $l$ through 0, let us assign the index 0 to this line. 

If $l = \mathbb{F}_q$, then an additive shift of $l$ is given by $\{c+j\alpha~|~c \in \mathbb{F}_q\}$ for some fixed $j \in \mathbb{F}_q$; we assign the index $j$ to this line. 

If $l = \{c(i+\alpha) ~|~c \in \mathbb{F}_q\}$ for some $i \in \mathbb{F}_q$, then an additive shift of $l$ is given by $\{j+c(i+\alpha)~|~c \in \mathbb{F}_q\}$ for some $j \in \mathbb{F}_q$; we assign the index $j$ to this line.

Thus, we have defined an index for every line in the affine plane (indexes for distinct parallel lines are different). Let us define an entry of the array lying in the row indexed by a line $l$ through 0 and the column indexed by an element $x+y\alpha$. We put the value of this entry to be the index of the line that is parallel to $l$ and contains the element $x+y\alpha$. Since every point of the plane uniquely determines a pair of intersecting lines from any two fixed parallel classes, our construction gives an orthogonal array $OA(q+1,q)$.

\begin{lemma}\label{PaleyAsOAGraph}
Given an odd prime power $q$ and an integer $m>1$ such that $m$ divides $q+1$, the generalised Paley graph $\mathrm{GP}(q^2,m)$ is the block graph of an orthogonal array $OA(\frac{q+1}{m},q)$. Moreover, canonical cliques in both graphs match with each other.
\end{lemma}
\begin{proof}
In the above construction of the orthogonal array $OA(q+1,q)$, we take the $\frac{q+1}{m}$ rows that correspond to the slopes that are $m$-th powers in $\mathbb{F}_{q^2}^*$; these rows form an orthogonal array $OA(\frac{q+1}{m},q)$. The block graph of this array is isomorphic to the generalised Paley graph $\mathrm{GP}(q^2,m)$. It is straightforward to verify the correspondence between canonical cliques in the two graphs. We refer to \cite[Section 3]{AGLY22} for more details. $\square$
\end{proof}

\medskip

Lemma~\ref{PaleyAsOAGraph} allows us to translate results in Section~\ref{BGOA} in the context of generalised Paley graphs. First, it indicates that it is not interesting to study maximal cliques in $\mathrm{GP}(q^2,\frac{q+1}{2})$ and $\mathrm{GP}(q^2, q+1)$ by Lemma~\ref{TSziklai}.

\begin{corollary}\label{(q+1)/2}
In the generalised Paley graph $\mathrm{GP}(q^2,\frac{q+1}{2})$ and $\mathrm{GP}(q^2, q+1)$, each maximal clique is maximum.
\end{corollary}
\begin{proof}
It follows from Lemma~\ref{OAmaximal} and Lemma~\ref{PaleyAsOAGraph}. Alternatively, note that $\mathrm{GP}(q^2, q+1)=\operatorname{Cay}(\mathbb{F}_{q^2}^+,\mathbb{F}_{q}^*)$ is known as a subfield example of cyclotomic strongly regular graphs, and it is easy to classify its cliques. $\square$
\end{proof}

\medskip

Next we deduce the weight-distribution bound for non-principal eigenvalues of $\mathrm{GP}(q^2,m)$.

\begin{corollary}\label{WDBGP}
Given an odd prime power $q$ and an integer $m>1$ such that $m$ divides $q+1$, the following statements hold.\\
{\rm (1)} A $(-\frac{q+1}{m})$-eigenfunction of the generalised Paley graph $\mathrm{GP}(q^2,m)$ has at least $\frac{2(q+1)}{m}$ non-zeroes.\\
{\rm (2)} A $\frac{(m-1)q-1}{m}$-eigenfunction of the generalised Paley graph $\mathrm{GP}(q^2,m)$ has at least $2\frac{m-1}{m}(q+1)$ non-zeroes.
\end{corollary}
\begin{proof}
It follows from Lemma \ref{WDBOA} and Lemma \ref{PaleyAsOAGraph}. $\square$
\end{proof}
\medskip

Finally we derive an upper bound on the size of maximal cliques which are not maximum.

\begin{proposition}\label{prop:upperbound}
Let $m\mid (q+1)$ and $2\leq m \leq \frac{q+1}{3}$. If $C'$ is a maximal clique in the generalised Paley graph $\mathrm{GP}(q^2,m)$ which is not maximum, then $|C'|\leq (\frac{q+1}{m}-1)^2$.
\end{proposition}

\begin{proof}
Let $w=(q+1)/m$. By our assumption, $w \geq 3$. By Lemma~\ref{PaleyAsOAGraph}, the generalised Paley graph $\mathrm{GP}(q^2,m)$ can be realized as the block graph $X_O$ of an orthogonal array $O=OA(w,q)$. Moreover, canonical cliques in both graphs match with each other. Recall Lemma~\ref{TSziklai} states that the only maximum cliques in $\mathrm{GP}(q^2,m)$ are those canonical cliques. Thus, if $C'$ is a maximal clique in $\mathrm{GP}(q^2,m)$ which is not maximum, then the corresponding clique $C$ in the block graph $X_{O}$ is a non-canonical maximal clique. Thus, Lemma~\ref{OAmaximal} implies that $|C'|=|C| \leq (w-1)^2$. $\square$
\end{proof}

\medskip

We remark that the upper bound $(\frac{q+1}{m}-1)^2$ can be attained; see Example~\ref{counterexample} for an infinite family of such generalised Paley graphs. Under extra assumptions, Proposition~\ref{prop:upperbound} could be improved, and we discuss such improvements in Section~\ref{stability}.

\subsection{The subgroup $Q$ of order $q+1$ in $\mathbb{F}_{q^2}^*$}
For any $\gamma = x+y\alpha\in \mathbb{F}_{q^2}^*$, where $x,y \in \F_q$, define the \emph{norm} mapping $N: \mathbb{F}^*_{q^2} \to \mathbb{F}^*_{q}$ by 
$$N(\gamma) := \gamma^{q+1} = \gamma\gamma^{q} = (x+y\alpha)(x-y\alpha) = x^2 - y^2d.$$
Note that $N$ is precisely the norm with respect to the quadratic field extension $\F_{q^2}/\F_q$. 

Observe that
$N$ is a homomorphism with $\operatorname{Im}(N) = \mathbb{F}^*_q$ and 
$$\operatorname{Ker}(N) = \big\{\, x+y\alpha\in \mathbb{F}_{q^2}^*~\big|~x^2 - y^2d = 1 \,\big\}.$$
The first isomorphism theorem for groups implies that $\operatorname{Ker}(N)$ is a subgroup in $\mathbb{F}_{q^2}^*$ of order $q+1$. In addition, it is defined by a quadratic equation, so its elements form an oval in $AG(2,q)$.

Let $\beta$ be a primitive element in $\mathbb{F}_{q^2}$ and $\omega = \beta^{q-1}$. 
Then $Q = \langle\omega\rangle$ is a subgroup of order $q+1$ in $\mathbb{F}_{q^2}^*$.
On the other hand, $\operatorname{Ker}(N)$ is also a subgroup of order $q+1$ in $\mathbb{F}_{q^2}^*$.
Since the subgroups of a cyclic group are uniquely determined by the divisors of the group order, $Q$ and $\operatorname{Ker}(N)$ coincide. Thus, we have the following lemma.

\begin{lemma}\label{norm}
For each $x+y\alpha \in Q$, we have $x^2-y^2d=1$.
\end{lemma}

Since $m \mid (q+1)$, we can divide the oval $Q=\operatorname{Ker}(N)=\langle\omega\rangle$ into $m$ parts
$$
Q = Q_0 \cup Q_1 \cup \ldots \cup Q_{m-1},
$$
where $Q_0 = \langle\omega^{m}\rangle$, and $Q_1 = \omega Q_0$, \ldots, $Q_{m-1} = \omega^{m-1}Q_0$ are cosets of $Q_0$. Note that $1$ belongs to $Q_0$, and the tangent to $Q$ at 1 is $\{1+c\alpha~|~c \in \mathbb{F}_q\}$, which, in view of Lemma \ref{calpha}, is an $m$-ary line if and only if $m$ divides $\frac{q+1}{2}$.

The following lemma gives a correspondence between $m$-th powers in $\F_{q^2}^*$ and $m$-th powers in the oval $Q$.

\begin{lemma}\label{hom} 
Let $q$ be a prime power and $m \geq 2$ be an integer such that $m$ divides $q+1$. 
The mapping $\gamma\mapsto\gamma^{q-1}$ is a homomorphism from $\mathbb{F}_{q^2}^*$ to $Q$. Moreover, 
an element $\gamma$ is an $m$-th power in $\mathbb{F}_{q^2}^*$ if and only if $\gamma^{q-1}$ is an $m$-th power in $Q$.
\end{lemma}
\begin{proof} 
Note that for any $\gamma \in \F_{q^2}^*, \gamma^{q-1} \in \langle \beta^{q-1} \rangle=\langle \omega \rangle=Q$. For any $\gamma_1,\gamma_2 \in \mathbb{F}_{q^2}^*$, $(\gamma_1\gamma_2)^{q-1} = \gamma_1^{q-1}\gamma_2^{q-1}$, so the mapping is a homomorphism from $\mathbb{F}_{q^2}^*$ to $Q$. 

Since $m$ divides $|\mathbb{F}_{q^2}^*|=q^2-1 = (q-1)(q+1)$, there are exactly $\frac{q^2-1}{m}$ $m$-th powers in $\mathbb{F}_{q^2}^*$; they form the set $\langle\beta^m\rangle$. Also, since $m$ divides $q+1$, there are exactly $\frac{q+1}{m}$ $m$-th powers in $Q$; they form the set $Q_0$. 

Let $\beta^{ms}$ be an element from $\langle\beta^m\rangle$. Then 
$$(\beta^{ms})^{q-1} = (\beta^{q-1})^{ms} = (\omega^s)^m.$$
Thus, the image of an $m$-th power from $\mathbb{F}_{q^2}^*$ is an $m$-th power in $Q$. Let us consider any $i \in \{1,\ldots, m-1\}$ and the element $\beta^{ms+i}$, which is not an $m$-th power in $\mathbb{F}_{q^2}^*$. Computing its image, we have
$$(\beta^{ms+i})^{q-1} = (\beta^{q-1})^i(\beta^{q-1})^{ms} = \omega^i(\omega^s)^m,$$
which is not an $m$-th power in $Q$.
$\square$
\end{proof}

\medskip

The following lemma determines the neighbourhood of 0 in $Q$ in $\mathrm{GP}(q^2,m)$
\begin{lemma}\label{NeighboursOf0InQ}
The following statements hold in the graph $\mathrm{GP}(q^2,m)$.\\
{\rm (1)} If $m$ is odd, then the neighbours of $0$ in $Q$ are exactly the elements of $Q_0$.\\
{\rm (2)} If $m$ is even, then the neighbours of $0$ in $Q$ are exactly the elements of $Q_0 \cup Q_{\frac{m}{2}}$.
\end{lemma}
\begin{proof}
Let $\omega^t$ be an arbitrary element from $Q$, where $t = me+i$, $0 \le i \le m-1$ (more precisely, $\omega^t$ represents an arbitrary element from $Q_i$). Since $\omega = \beta^{q-1}$, we have 
$$
\omega^t = \beta^{t(q-1)} = \frac{\beta^{t(q+1)}}{\beta^{-2t}}.
$$
Since $m$ divides $q+1$, $\omega^t$ is an $m$-th power in $\mathbb{F}_{q^2}^*$ if and only if $m$ divides $2t$, or equivalently, $m$ divides $2i$. Since $0 \le i \le m-1$, we conclude that $\omega^t$ is an $m$-th power in $\mathbb{F}_{q^2}^*$ if and only if $i = 0$ or, additionally, $i = \frac{m}{2}$ when $m$ is even.
$\square$
\end{proof}

\medskip
The following lemma is crucial for understanding the adjacency structure of the oval $Q$ in the graph $\mathrm{GP}(q^2,m)$. 
\begin{lemma}\label{gammaminus1}
Let $\gamma = x+y\alpha$ be an arbitrary element from $Q$, $\gamma \ne 1$. Then, for the image of $(\gamma-1)$ under the action of the homomorphism from Lemma \ref{hom}, the equality
$$(\gamma-1)^{q-1} = -\frac{1}{\gamma}$$ holds.
\end{lemma}
\begin{proof}
Recall that $\alpha^2=d$ and $\alpha^{q}=-\alpha$.
We have
\begin{align*}
(\gamma-1)^{q-1} 
&= (x-1+y\alpha)^{q-1} = \frac{(x-1+y\alpha)^{q}}{(x-1+y\alpha)} = 
\frac{x-1-y\alpha}{x-1+y\alpha} \\
&=\frac{(x-y\alpha-1)(x+y\alpha+1)}{(x+y\alpha-1)(x+y\alpha+1)}
=\frac{x^2-(y\alpha+1)^2}{(x+y\alpha)^2-1}
=\frac{x^2-y^2d-2y\alpha-1}{x^2+2xy\alpha+y^2d-1}.    
\end{align*}
By Lemma~\ref{norm}, $x^2-y^2d = 1$. Therefore, $$
(\gamma-1)^{q-1}=\frac{x^2-y^2d-2y\alpha-1}{x^2+2xy\alpha+y^2d-1}
=
\frac{-2y\alpha}{2y^2d+2xy\alpha}
=
\frac{-2y\alpha}{2y\alpha(y\alpha+x)} = 
-\frac{1}{\gamma}.
~\square
$$
\end{proof}

Let $\omega^{t_1}$ and $\omega^{t_2}$ be two arbitrary elements from $Q$ such that $0 \le t_1 < t_2 \le q$ so that $\omega^{t_1}$ and $\omega^{t_2}$ are distinct. The following lemma shows that the adjacency between $\omega^{t_1}$ and $\omega^{t_2}$ is determined by the remainder of $t_1+t_2$ modulo $m$. 

\begin{lemma}\label{t1t2}
The elements $\omega^{t_1}$ and $\omega^{t_2}$ are adjacent in $\mathrm{GP}(q^2,m)$ if and only if 
$$\frac{q+1}{2} - (t_1 + t_2) \equiv 0\pmod m.$$
\end{lemma}
\begin{proof}
We compute the image of $\omega^{t_2} - \omega^{t_1}$ under the action of the homomorphism from Lemma \ref{hom} and apply Lemma \ref{gammaminus1}: 
\begin{align}\label{eqomega}
(\omega^{t_2} - \omega^{t_1})^{q-1} 
&=(\omega^{t_1})^{q-1} (\omega^{t_2-t_1} - 1)^{q-1} 
=\omega^{t_1(q-1)}\frac{-1}{\omega^{t_2-t_1}}=-\omega^{t_1(q-1)-t_2+t_1} \nonumber \\
&=-\omega^{t_1(q+1)-t_1-t_2}
=-\omega^{-t_1-t_2}
=\omega^{\frac{q+1}{2}}\omega^{-t_1-t_2}
=\omega^{\frac{q+1}{2}-(t_1+t_2)}. 
\end{align}
Thus, by Lemma \ref{hom}, the elements $\omega^{t_2}$ and $\omega^{t_1}$ are adjacent in $\mathrm{GP}(q^2,m)$ if and only if $\omega^{\frac{q+1}{2}-(t_1+t_2)}$ is an $m$-th power in $Q$, which is equivalent to 
$$
\frac{q+1}{2}-(t_1+t_2) \equiv 0\pmod m.
~\square
$$
\end{proof}

\section{Constructions of cliques in generalised Paley graphs of square order} \label{sec: construction}

We have mentioned in Section~\ref{wdb} that there is a close connection between the tightness of the weight-distribution bound for the eigenfunction and the substructure of the graph. Motivated by this connection, we discover a rich geometric structure in generalised Paley graphs. In this section, we construct new families of cliques in the generalised Paley graph $\mathrm{GP}(q^2,m)$, where $m \ge 2$ and $m \mid (q+1)$. Recall that $Q = \langle\omega\rangle$ is an oval, where $\beta$ is a fixed primitive element of $\F_{q^2}$ and $\omega=\beta^{q-1}$. Also recall that that the elements of $\mathbb{F}_q^* = \langle\beta^{q+1}\rangle$ are $m$-th powers in $\mathbb{F}_{q^2}^*$ since $m \mid (q+1)$. 

\subsection{Structure of the oval $Q$ in $\mathrm{GP}(q^2,m)$}\label{sec4.1}
Let $t_1 = me_1+i_1$ and $t_2 = me_2+i_2$, where $i_1$ and $i_2$ are the remainders after division of $t_1$ and $t_2$ by $m$ (i.e. $0 \le i_1 \le m-1$ and $0 \le i_2 \le m-1$). Note that $\omega^{t_1}$ and $\omega^{t_2}$ belong to $Q_{i_1}$ and $Q_{i_2}$, respectively. Actually, Lemma \ref{t1t2} can be used to show that there are either all possible edges between $Q_{i_1}$ and $Q_{i_2}$ or no such edges.
Also, Lemma \ref{t1t2} can be used to show that each of the sets $Q_0, Q_1, \ldots, Q_{m-1}$ is either a clique or an independent set.

\begin{theorem}\label{StructureOfQ}
Given an odd prime power $q$ and an integer $m \ge 2$ such that $m$ divides $q+1$, the following statements hold for the subgraph of $\mathrm{GP}(q^2,m)$ induced by $Q$.\\
{\rm (1)} Assume that $m$ divides $\frac{q+1}{2}$. Then for any distinct $i_1, i_2$ such that $0 \le i_1 < i_2 \le m-1$, there are all possible edges between the sets $Q_{i_1}$ and $Q_{i_2}$ if $$i_1+i_2 \equiv 0\pmod m,$$ and there are no such edges if $$i_1+i_2 \not\equiv 0\pmod m.$$
{\rm (1.1)} Assume further that $m$ is odd. Then $Q_0$ is a clique, $Q_1,\ldots, Q_{m-1}$ are independent sets. In particular, $Q_1 \cup Q_{m-1}, \ldots, Q_{\frac{m-1}{2}} \cup Q_{\frac{m+1}{2}}$ induce $\frac{m-1}{2}$ complete bipartite graphs.\\
{\rm (1.2)} Assume further that $m$ is even. Then $Q_0, Q_{\frac{m}{2}}$ are cliques, and $Q_1,\ldots, Q_{\frac{m}{2}-1},$ $Q_{\frac{m}{2}+1},\ldots,Q_{m-1}$ are independent sets. In particular, if $m \geq 4$, then $Q_1 \cup Q_{m-1}, \ldots, Q_{\frac{m}{2}-1} \cup Q_{\frac{m}{2}+1}$ induce $(\frac{m}{2}-1)$ complete bipartite graphs.\\
{\rm (2)} Assume that $m$ does not divide $\frac{q+1}{2}$. Then $m$ is even, and for any distinct $i_1, i_2$ such that $0 \le i_1 < i_2 \le m-1$, there are all possible edges between the sets $Q_{i_1}$ and $Q_{i_2}$ if $$i_1+i_2 \equiv \frac{m}{2}\pmod m,$$ and there are no such edges if $$i_1+i_2 \not\equiv \frac{m}{2}\pmod m.$$
{\rm (2.1)} Assume further that $\frac{m}{2}$ is odd. Then $Q_0, Q_1,\ldots, Q_{m-1}$ are independent sets. In particular, if $m = 2$, then $Q = Q_0 \cup Q_1$ induces a complete bipartite graph; if $m \ge 6$, then
$$Q_0 \cup Q_{\frac{m}{2}}, \ldots, Q_{\frac{m-2}{4}} \cup Q_{\frac{m+2}{4}}, Q_{\frac{m}{2}+1} \cup Q_{m-1}, \ldots, Q_{\frac{3m-2}{4}} \cup Q_{\frac{3m+2}{4}}$$ induce $\frac{m}{2}$ complete bipartite graphs.\\
{\rm (2.2)} Assume further that $\frac{m}{2}$ is even. Then
$Q_{\frac{m}{4}}, Q_{\frac{3m}{4}}$ are cliques, and $Q_0,\ldots, Q_{\frac{m}{4}-1},$
$Q_{\frac{m}{4}+1},\ldots, Q_{\frac{3m}{4}-1},
Q_{\frac{3m}{4}+1},\ldots,Q_{m-1}$ are independent sets.
In particular, if $m = 4$, $Q_0 \cup Q_2$ induces a complete bipartite graph; if $m \ge 8$, then 
$$Q_0 \cup Q_{\frac{m}{2}}, \ldots, Q_{\frac{m-4}{4}} \cup Q_{\frac{m+4}{4}}, Q_{\frac{m}{2}+1} \cup Q_{m-1}, \ldots, Q_{\frac{3m-4}{4}} \cup Q_{\frac{3m+4}{4}}$$ induce $\frac{m-2}{2}$ complete bipartite graphs.
\end{theorem}

\begin{proof}
We consider two arbitrary distinct elements $\omega^{t_1}, \omega^{t_2} \in Q$, where $t_1 = me_1+i_1$, $t_2 = me_2+i_1$, $0 \le i_1 \le m-1$ and $0 \le i_2 \le m-1$, which means that $\omega^{t_1}$ and $\omega^{t_2}$ represent arbitrary elements from $Q_{i_1}$ and $Q_{i_2}$, respectively. By Lemma \ref{t1t2}, the elements $\omega^{t_1}$ and $\omega^{t_2}$ are adjacent in $\mathrm{GP}(q^2,m)$ if and only if 
$
\frac{q+1}{2}-(t_1+t_2) \equiv 0\pmod m
$
or, equivalently,
\begin{equation}\label{Eqmodm}
\frac{q+1}{2}-(i_1+i_2) \equiv 0\pmod m.    
\end{equation}
Note, for any $i \in \{0,1,\ldots,m-1\}$, the subgraph induced by $Q_i$ is a clique if and only if equation (\ref{Eqmodm}) holds in the case $i_1 = i_2 = i$.

(1) Since $m$ divides $\frac{q+1}{2}$, equation (\ref{Eqmodm}) is equivalent to 
\begin{equation}\label{case12}
i_1+i_2 \equiv 0\pmod m.
\end{equation}

(1.1) Suppose $m$ is odd. The only case when $i_1 = i_2$ and equation (\ref{case12}) hold, is the case $i_1 = i_2 = 0$. Thus, $Q_0$ is a clique and $Q_1,\ldots, Q_{m-1}$ are independent sets, and Equation (\ref{case12}) implies that $Q_1 \cup Q_{m-1}, \ldots, Q_{\frac{m-1}{2}} \cup Q_{\frac{m+1}{2}}$ induce $\frac{m-1}{2}$ complete bipartite graphs, and there are no other edges in $Q$.

(1.2) Suppose $m$ is even. The only two cases when $i_1 = i_2$ and equation (\ref{case12}) hold, are the cases $i_1 = i_2 = 0$ and $i_1 = i_2 = \frac{m}{2}$.
Thus, $Q_0, Q_{\frac{m}{2}}$ are cliques, $Q_1,\ldots, Q_{\frac{m}{2}-1},Q_{\frac{m}{2}+1},\ldots,Q_{m-1}$ are independent sets, and equation (\ref{case12}) implies that $Q_1 \cup Q_{m-1}, \ldots, Q_{\frac{m}{2}-1} \cup Q_{\frac{m}{2}+1}$ induce $(\frac{m}{2}-1)$ complete bipartite graphs, and there are no other edges in $Q$.

(2) Since $m$ divides $q+1$ and does not divide $\frac{q+1}{2}$, $m$ is even. 
It follows that $\frac{q+1}{2} \equiv \frac{m}{2} \pmod m$. Thus, equation (\ref{Eqmodm}) is equivalent to 
\begin{equation}\label{case3} 
i_1+i_2 \equiv \frac{m}{2}\pmod m.    
\end{equation}

(2.1) Suppose $\frac{m}{2}$ is odd. Then equation (\ref{case3}) has no solution of the form $i_1 = i_2$, which means that $Q_0, Q_1,\ldots, Q_{m-1}$ are independent sets. Since $0 \le i_1 < i_2 \le m-1$,
equation (\ref{case3}) implies that either
\begin{equation}\label{mdiv2}
i_1+i_2 = \frac{m}{2}    
\end{equation}
 or 
\begin{equation}\label{3mdiv2}
i_1+i_2 = \frac{3m}{2}    
\end{equation}
holds. Note that the solutions to equation (\ref{mdiv2}) satisfy $0\le i_1 < i_2 \le \frac{m}{2}$; there are $\frac{m+2}{4}$ such solutions. On the other hand, the solutions of equation (\ref{3mdiv2}) satisfy $\frac{m}{2} < i_1 < i_2 \le m-1$; there are $\frac{m-2}{4}$ such solutions. Thus, if $m = 2$, $Q = Q_0 \cup Q_1$ induces a complete bipartite graph; if $m \ge 6$,
$Q_0 \cup Q_{\frac{m}{2}}, \ldots, Q_{\frac{m-2}{4}} \cup Q_{\frac{m+2}{4}}, Q_{\frac{m}{2}+1} \cup Q_{m-1}, \ldots, Q_{\frac{3m-2}{4}} \cup Q_{\frac{3m+2}{4}}$ induce $\frac{m}{2}$ complete bipartite graphs, and there are no other edges in $Q$.

(2.2) Suppose $\frac{m}{2}$ is even. 
Then equation (\ref{case3}) has two solutions of the form $i_1 = i_2$, these are $i_1 = i_2 = \frac{m}{4}$ and $i_1 = i_2 = \frac{3m}{4}$, which means that $Q_{\frac{m}{4}}$, $Q_{\frac{3m}{4}}$ are cliques; for $m=4$, $Q_0, Q_2$ are independent sets; for $m \ge 8$, $Q_0,\ldots, Q_{\frac{m}{4}-1},
Q_{\frac{m}{4}+1},\ldots, Q_{\frac{3m}{4}-1},
Q_{\frac{3m}{4}+1},\ldots,Q_{m-1}$ are independent sets.
%Since $0 \le i_1 < i_2 \le m-1$, equation (\ref{case3}) implies that equation (\ref{mdiv2}) or equation (\ref{3mdiv2}) hold. Note that the solutions to rquation (\ref{mdiv2}) satisfy $0\le i_1 < i_2 \le \frac{m}{2}$; there are $\frac{m}{4}$ such solutions. On the other hand, the solutions of rquation (\ref{3mdiv2}) satisfy $\frac{m}{2} < i_1 < i_2 \le m-1$; there are $\frac{m-4}{4}$ such solutions. 
Similar to the case (3.1), if $m = 4$, then $Q_0 \cup Q_2$ induces a complete bipartite graph; if $m \ge 8$, then 
$Q_0 \cup Q_{\frac{m}{2}}, \ldots, Q_{\frac{m-4}{4}} \cup Q_{\frac{m+4}{4}}, Q_{\frac{m}{2}+1} \cup Q_{m-1}, \ldots, Q_{\frac{3m-4}{4}} \cup Q_{\frac{3m+4}{4}}$ induce $\frac{m-2}{2}$ complete bipartite graphs, and there are no other edges in $Q$. $\square$
\end{proof}

\medskip

Note that Theorem \ref{StructureOfQ} generalises \cite[Lemma 8]{GKSV18}. In Section \ref{TightnessWDB}, we use Theorem \ref{StructureOfQ} to define eigenfunctions of $\mathrm{GP}(q^2,m)$ corresponding to the smallest eigenvalue $(-\frac{q+1}{m})$ that meet the weight-distribution bound from Lemma \ref{WDBGP}. Our goal is to show a general connection between maximal cliques in $\mathrm{GP}(q^2,m)$ and $(-\frac{q+1}{m})$-eigenfunctions that meet the weight-distribution bound. Theorem \ref{StructureOfQ} shows that the only case when all $Q_0, \ldots, Q_{m-1}$ are cliques is the case $m = 2$ and $2$ divides $\frac{q+1}{2}$ (equivalently, $m=2$ and $q \equiv 3 \pmod 4$); in other cases we have two, or one, or no cliques among $Q_0, \ldots, Q_{m-1}$; and there is a perfect matching (in the sense of complete bipartite graphs) on the sets $\{Q_i~|~ i \in \{0,\ldots, m-1\}, Q_i\text{~is an independent set}\}$. Thus, we have the following corollary.

\begin{corollary}\label{CompleteBipartiteSubgraph}
Let $q$ be an odd prime power and $m \ge 2$ be an integer such that $m$ divides $q+1$. Then,
except for the case $m = 2$ and $2$ divides $\frac{q+1}{2}$, there is at least one pair $Q_{i_1}, Q_{i_2}$ among $Q_0, \ldots, Q_{m-1}$ such that $Q_{i_1} \cup Q_{i_2}$ induces a complete bipartite subgraph in $\mathrm{GP}(q^2,m)$.
\end{corollary}

\subsection{The tightness of the weight-distribution bound for the smallest eigenvalue $-\frac{q+1}{m}$ of  $\mathrm{GP}(q^2,m)$}\label{TightnessWDB}
The weight-distribution bound (Corollary~\ref{WDBGP}) states that an eigefunction of the Paley graph $\mathrm{GP}(q^2,2)$ corresponding to a non-principal eigenvalue has at least $q+1$ non-zeroes; in \cite[Theorem 2]{GKSV18}, it was proved that this bound is tight. In this section we show that the weight-distribution bound is tight for the smallest eigenvalue $-\frac{q+1}{m}$ of $\mathrm{GP}(q^2,m)$ for every odd prime power $q$ and integer $m \ge 2$, where $m$ divides $q+1$. We remark that this would follow immediately from Corollary~\ref{CompleteBipartiteSubgraph} and a consequence of the main results in \cite{KMP16}; see Remark~\ref{connection}. For the sake of completeness, we provide a short proof.

In view of Corollary \ref{CompleteBipartiteSubgraph},
except for the case $m = 2$ and $2$ divides $\frac{q+1}{2}$, there is at least one pair $Q_{i_1}, Q_{i_2}$ among $Q_0, \ldots, Q_{m-1}$ such that $Q_{i_1} \cup Q_{i_2}$ induces a complete bipartite subgraph.
Define a function $f_{Q_{i_1},Q_{i_2}}:\mathbb{F}_{q^2}\to \mathbb{R}$ by the following rule. For any $\gamma \in \mathbb{F}_{q^2}$, define
$$
f_{Q_{i_1},Q_{i_2}}(\gamma):=
\begin{cases}
1 & \text{if~~}  \gamma \in Q_{i_1},\\
-1 & \text{if~~} \gamma \in Q_{i_2},\\
0 & \text{otherwise.}
\end{cases}
$$

\begin{theorem}\label{thm: WDBtight}
Let $m \geq 3$ and $m \mid (q+1)$.
The function $f_{Q_{i_1},Q_{i_2}}$ is a $(-\frac{q+1}{m})$-eigenfunction of $\mathrm{GP}(q^2,m)$, which has cardinality of support $\frac{2(q+1)}{m}$ and thus meets the weight-distribution bound in view of Corollary~\ref{WDBGP}.
\end{theorem}
\begin{proof}
Let us consider a vertex $v \in \mathbb{F}_{q^2} \setminus (Q_{i_1} \cup Q_{i_2})$; then we have $f_{Q_{i_1},Q_{i_2}}(v)=0$. We consider the following three cases:
\begin{itemize}
    \item $v=0$. In view of Lemma~\ref{NeighboursOf0InQ} and Corollary~\ref{CompleteBipartiteSubgraph}, the vertex $0$ either has no neighbours in $Q_{i_1} \cup Q_{i_2}$ or is adjacent to every vertex from $Q_{i_1} \cup Q_{i_2}$ (in the case $\{i_1,i_2\} = \{0,\frac{m}{2}\}$).
    \item $v \in Q_i$, where $i \in \{0,\ldots,m-1\}\setminus \{i_1,i_2\}$.  In view of Theorem \ref{StructureOfQ}, the vertices from $Q_i$ have no neighbours in $Q_{i_1} \cup Q_{i_2}$. 
    \item $v \in \mathbb{F}_{q^2} \setminus (Q \cup \{0\})$. If $v$ has a neighbour $w$ in $Q_{i_1}$ (resp. $Q_{i_2}$), then $v$ lies on the $m$-ary line passing through $v$ and $w$, and thus intersects $Q_{i_2}$ (resp. $Q_{i_1}$). In fact, every $m$-ary line through $w$ intersects $Q_{i_2}$ (resp. $Q_{i_1}$) exactly once, since $Q$ is an oval, $Q_{i_1} \cup Q_{i_2}$ is a complete bipartite graph with part of size $\frac{q+1}{m}$, and the number of $m$-ary lines passing through $w$ is precisely $\frac{q+1}{m}$ by Lemma~\ref{lem: numberoflines}. 
Therefore, the neighbours of $v$ in $Q_{i_1}$ and $Q_{i_2}$ come in pairs; in other words, the contribution from $Q_{i_1}$ and the contribution from $Q_{i_2}$ cancel out with each other. Thus, the condition (\ref{LocalCondition}) is satisfied for every vertex with value 0 of $f_{Q_{i_1},Q_{i_2}}$.
\end{itemize} 

Finally, for the vertices from $Q_{i_1} \cup Q_{i_2}$, condition (\ref{LocalCondition}) is satisfied since $Q_{i_1} \cup Q_{i_2}$ induces a complete bipartite graph with parts $Q_{i_1}$ and $Q_{i_2}$ of size $\frac{q+1}{m}$. $\square$
\end{proof}

\medskip

Combining the known results for Paley graphs \cite[Theorem 2]{GKSV18}, we conclude that the weight-distribution bound is tight for the eigenvalue $-\frac{q+1}{m}$ for any generalised Paley graph $\mathrm{GP}(q^2,m)$, where $m>1$ and $m \mid (q+1)$.

\begin{remark}\rm \label{connection}
In general, one can show that for a primitive strongly regular graph $X$ with eigenvalues $\theta_1$ and $\theta_2$ such that $\theta_2 < 0 < \theta_1$, the weight-distribution bound for the eigenvalue $\theta_2$ is tight if and only if there exists a complete bipartite induced subgraph in $X$ with parts $T_0$ and $T_1$ of size $-\theta_2$. This fact can be proved using the main results in \cite[Section 2.5]{KMP16} and properties of strongly regular graphs; see also the discussion in \cite[Section 4.3]{SV21}. Using this fact, Theorem~\ref{thm: WDBtight} follows from Corollary~\ref{CompleteBipartiteSubgraph} immediately. 

\end{remark}

\begin{remark}\rm
The weight-distribution bound is tight for both non-principal eigenvalues for Paley graphs of square order since Paley graphs are self-complementary. In general, when $m \geq 3$, we suspect that weight-distribution bound is not tight for the other non-principal eigenvalue $\frac{(m-1)q-1}{m}$ of $\mathrm{GP}(q^2,m)$. 
\end{remark}

\subsection{{\rm (}$\mathbb{F}_q,\alpha{\rm )}$-construction and {\rm(}$\alpha\mathbb{F}_q,1 {\rm)}$-construction of cliques in $\mathrm{GP}(q^2,m)$}

Recall that we defined $\alpha$ so that $\F_{q^2}=\F_q \oplus \alpha \F_q$ and identify $\F_{q^2}$ as the affine plane $AG(2,q)$.

In \cite{BEHW96}, a construction of maximal but not maximum cliques in Paley graphs of square order was proposed. In \cite{GMS22}, this construction was reinterpreted in terms of quadratic and non-quadratic lines (see \cite[Construction 1 and Construction 2]{GMS22}).
In this section we propose an extension of this construction for $\mathrm{GP}(q^2,m)$ in terms of the $m$-ary lines. The generalised construction also requires choosing an $m$-ary line and a point that is not on the line. In view of Lemma \ref{AutGP}, without loss of generality, we may choose any $m$-ry line (it is convenient to choose $m$-ary lines for which we can decide in general if they are $m$-ary, i.e. the lines $\mathbb{F}_q$ and $\alpha\mathbb{F}_q$; see Lemma \ref{calpha}). In view of Lemma \ref{AutGPSubgroup}, without loss of generality, we may choose the line $\F_q$ and any point outside $\F_q$.

Let us consider the $m$-ary line $\mathbb{F}_q$, the point $\alpha \notin \mathbb{F}_q$, and the pencil of $m$-ary lines through $\alpha$.
Each line of the pencil, except $\{c+\alpha~|~c \in \mathbb{F}_q\}$, intersects $\mathbb{F}_q$ at a point. In view of Lemma~\ref{lem: numberoflines}, there are $\frac{q+1-m}{m}$ such intersection points and we denote them by $c_1, \ldots, c_{\frac{q+1-m}{m}}$. So, the set $\{\alpha, c_1, \ldots, c_{\frac{q+1-m}{m}}\}$ is a clique of size $\frac{q+1}{m}$ in $\mathrm{GP}(q^2,m)$.
The points $-\alpha$ and $\alpha$ have the same neighbours in $\mathbb{F}_q$ since the sets
$\{\alpha, c_1, \ldots, c_{\frac{q+1-m}{m}}\}$
and
$\{-\alpha, c_1, \ldots, c_{\frac{q+1-m}{m}}\}$
are equivalent under the field automorphism 
$\gamma \mapsto \gamma^q$, an analogue of the complex conjugation, 
which acts as $\alpha \mapsto -\alpha$ and $c_k \mapsto c_k$. 
Therefore, the set $\{-\alpha, c_1, \ldots, c_{\frac{q+1-m}{m}}\}$ is also a clique. 
If $m$ divides $\frac{q+1}{2}$,  
in view of Lemma \ref{calpha}, the line $\{c\alpha~|~c \in \mathbb{F}_q\}$, which contains $-\alpha$ and $\alpha$, is $m$-ary. 
Hence, the cliques 
$\{\alpha, c_1, \ldots, c_{\frac{q+1-m}{m}}\}$ and $\{-\alpha, c_1, \ldots, c_{\frac{q+1-m}{m}}\}$
are combined into one. We summarise the above discussion in the following proposition. 

\begin{proposition}[{\rm (}$\mathbb{F}_q,\alpha{\rm )}$-construction]\label{Fqalpha}
    Let $q$ be an odd prime power and $m \ge 2$ be an integer, $m$ divides $q+1$.
    If $m$ does not divide $\frac{q+1}{2}$, then 
    \begin{center}
    $\{\alpha, c_1, \ldots, c_{\frac{q+1-m}{m}}\}$ 
    is a clique 
    \end{center}
    of size $\frac{q+1}{m}$ in $\mathrm{GP}(q^2,m)$; 
    if $m$ divides $\frac{q+1}{2}$, then 
    \begin{center}
    $\{\pm\alpha, c_1, \ldots, c_{\frac{q+1-m}{m}}\}$ is a clique 
    \end{center}
     of size $\frac{q+1+m}{m}$ in $\mathrm{GP}(q^2,m)$.
\end{proposition}

The second construction is similar to the construction from Proposition  
\ref{Fqalpha}, but deals only with the case when the line $\{c\alpha~|~c\in\mathbb{F}_q\}$ is $m$-ary (i.e., in view of Lemma \ref{calpha}, when $m$ divides $\frac{q+1}{2}$). Consider the line $\alpha \mathbb{F}_q$, the point $1 \notin \alpha \mathbb{F}_q$, and the pencil of $m$-ary lines through $1$.
They all, except $\{1+c\alpha\:|\:c \in \mathbb{F}_q\}$, intersect $\alpha\mathbb{F}_q$.
Denote the intersection points by $c_1\alpha, \ldots, c_{\frac{q+1-m}{m}}\alpha$.
In view of the automorphism 
$\gamma \mapsto -\gamma^q$, which stabilises $\alpha\mathbb{F}_q$ pointwise and swaps $-1$ and $1$, the elements $-1$ and $1$ have the same neighbours in $\alpha\mathbb{F}_q$.
Thus, $\{1, c_1\alpha, \ldots, c_{\frac{q+1-m}{m}}\alpha\}$ and $\{-1, c_1\alpha, \ldots, c_{\frac{q+1-m}{m}}\alpha\}$ are cliques of size $\frac{q+1}{m}$ in $\mathrm{GP}(q^2,m)$,
and they are combined into one since $-1$ and $1$ lie on the $m$-ary line $\mathbb{F}_q$ and thus are adjacent in $\mathrm{GP}(q^2,m)$. We summarise the above observation in the following:

\begin{proposition}[{\rm (}$\alpha\mathbb{F}_q,1{\rm )}$-construction]\label{alphaFq1}
 Let $q$ be an odd prime power and $m\ge 2$ be an integer such that $m \mid \frac{q+1}{2}$. Then
$$
\{\pm1, c_1\alpha, \ldots, c_{\frac{q+1-m}{m}}\alpha\} 
$$
forms a clique of size $\frac{q+1+m}{m}$ in $\mathrm{GP}(q^2,m)$.
\end{proposition}

In the case of Paley graphs, it is known that the $(\F_q, \alpha)$-construction gives a maximal clique (and so does the $(\alpha\F_q,1)$-construction; see \cite[Lemma 1]{GMS22}). Lemma~\ref{Paleymaximal} was first proved by Baker et al. \cite{BEHW96}. Sz{\H{o}}nyi outlined a different proof of Lemma~\ref{Paleymaximal} using character sums in \cite[Page 204]{Sz97}; see also a related discussion in \cite{KK09}. 

\begin{lemma} [\cite{BEHW96}]\label{Paleymaximal}
In the Paley graph of order $q^2$, the $(\F_q, \alpha)$ construction gives a maximal clique of size $\frac{1}{2}(q+1)$ or $\frac{1}{2}(q+3)$, accordingly as $q \equiv 1\pmod 4$ or $q \equiv 3 \pmod 4$.
\end{lemma}

Motivated by the above lemma, it is tempting to conjecture that in the generalised Paley graph $\mathrm{GP}(q^2,m)$, the cliques from the $(\F_q, \alpha)$-construction and the $(\alpha\F_q,1)$-construction are also maximal. However, it seems challenging to extend the Lemma~\ref{Paleymaximal} to the case of generalised Paley graphs. In fact, the naive extension is not always true and we have to refine it by adding the following two extra assumptions; note that in the case of Paley graphs, these two assumptions are automatically satisfied as $m=2$.

\begin{conjecture}\label{Conj1}
Assume that $p \nmid (m-1)$ and $2 \leq m \leq \frac{q+1}{3}$. Then the cliques from Proposition~\ref{Fqalpha} and Proposition~\ref{alphaFq1} are maximal.
\end{conjecture}

Using SageMath, we have verified Conjecture~\ref{Conj1} for all pairs $(q,m)$ for which $q \leq 500$ and $m \mid (q+1)$. We remark that the two assumptions in the statement of Conjecture~\ref{Conj1} cannot be dropped. We will see in Example~\ref{counterexample} that if $p \mid (m-1)$, then it is possible that these cliques are far from being maximal. In view of Corollary~\ref{(q+1)/2}, it is also necessary to avoid the cases $m=q+1$ and $m=\frac{q+1}{2}$. 

In Section~\ref{partialresult}, we will show that Conjecture~\ref{Conj1} is true under some extra assumptions in Corollary~\ref{cor: recover}, and show that the cliques constructed are nearly maximal in the general case in Theorem~\ref{thm: nearlymaximal}. In particular, we recover Lemma~\ref{Paleymaximal} as a special case. We will only focus on the maximality of the $(\F_q, \alpha)$-construction since it is evident that the two constructions are essentially equivalent.

\subsection{$Q_0$-construction and $\alpha Q_0$-construction of cliques in $\mathrm{GP}(q^2,m)$}

We first consider the construction of a clique in $\mathrm{GP}(q^2,m)$ based on $Q_0$.

\begin{proposition}[$Q_0$-construction]\label{0Q0Clique}
Let $q$ be an odd prime power and $m\ge 2$ be an integer such that $m \mid \frac{q+1}{2}$. Then 
$\{0\} \cup Q_0$ is a clique of size $\frac{q+1+m}{m}$ in $\mathrm{GP}(q^2,m)$.
\end{proposition}
\begin{proof}
By Theorem~\ref{StructureOfQ}, if $m$ divides $q+1$, then $Q_0$ is a clique. In view of Lemma \ref{NeighboursOf0InQ}, elements of $Q_0$ always lie in the neighbourhood of $0$. Thus, $\{0\} \cup Q_0$ is a clique of size $\frac{q+1+m}{m}$ in $\mathrm{GP}(q^2,m)$. $\square$
\end{proof}

\medskip
Note that the cliques from Proposition \ref{0Q0Clique} occur under the same assumptions as the cliques from Proposition \ref{alphaFq1}. However, as Theorem \ref{StructureOfQ} states, there are no cliques among $Q_0, \ldots, Q_{m-1}$ in the case when $m$ does not divide $\frac{q+1}{2}$ and $\frac{m}{2}$ is odd. In the rest of this section, we deal with the oval $\alpha Q$ and describe the adjacency structure of $\alpha Q$ in $\mathrm{GP}(q^2,m)$. In particular, it turns out that $\alpha Q_0$ is always a clique in $\mathrm{GP}(q^2,m)$. This result is a partial generalisation of \cite[Construction 4]{GMS22}.

Let $\omega^{t_1}$ and $\omega^{t_2}$ be two arbitrary elements from $Q$ such that $0 \le t_1 < t_2 \le q$; then $\alpha\omega^{t_1}$ and $\alpha\omega^{t_2}$ represent two arbitrary elements from $\alpha Q$.
The following lemma describes the adjacency between $\alpha\omega^{t_1}$ and $\alpha\omega^{t_2}$. 
\begin{lemma}\label{t1t2alpha}
The elements $\alpha\omega^{t_1}$ and $\alpha\omega^{t_2}$ are adjacent in $\mathrm{GP}(q^2,m)$ if and only if 
$$t_1 + t_2 \equiv 0\pmod m.$$
\end{lemma}
\begin{proof}
The proof is similar to the proof of Lemma~\ref{t1t2}. Recall that $\alpha^2=d$ is a non-square in $\F_q^*$, and thus $\alpha^{q-1}=d^{(q-1)/2}=-1$. By equation~\eqref{eqomega}, we have $$(\alpha\omega^{t_2} - \alpha\omega^{t_1})^{q-1}=-\omega^{\frac{q+1}{2}-(t_1+t_2)}=\omega^{-(t_1+t_2)},$$ and the conclusion follows. $\square$
\end{proof}
\begin{comment}
Consider the difference 
$$
\alpha\omega^{t_2} - \alpha\omega^{t_1} = \alpha\omega^{t_1}(\omega^{t_2-t_1} - 1).
$$
Compute the image under the action of the homomorphism from Lemma \ref{hom}, apply Lemma \ref{gammaminus1} and use the fact that $\beta^{\frac{q+1}{2}} = c\alpha$ for some $c \in \mathbb{F}_q^*$, which gives
\begin{align*}
(\alpha\omega^{t_1}(\omega^{t_2-t_1} - 1))^{q-1} 
&=\alpha^{q-1}(\omega^{t_1})^{q-1} (\omega^{t_2-t_1} - 1)^{q-1} \\
&=-\omega^{t_1(q-1)}\frac{-1}{\omega^{t_2-t_1}}
=\omega^{t_1(q-1)-t_2+t_1}\\
&=\omega^{t_1(q+1)-t_1-t_2}
=\omega^{-t_1-t_2}
=\frac{1}{\omega^{t_1+t_2}}.
\end{align*}
Thus, the elements $\alpha\omega^{t_2}$ and $\alpha\omega^{t_1}$ are adjacent in $\mathrm{GP}(q^2,m)$ if and only if $\omega^{(t_1+t_2)}$ is an $m$-th power in $Q$, which is equivalent to 
$$
t_1+t_2 \equiv 0\pmod m.
~\square
$$
\end{comment}

\medskip
%Let $t_1 = me_1+i_1$ and $t_2 = me_2+i_2$, where $i_1$ and $i_2$ are the remainders after division of $t_1$ and $t_2$ by $m$ (i.e. $0 \le i_1 \le m-1$ and $0 \le i_2 \le m-1$). Note that $\alpha\omega^{t_1}$ and $\alpha\omega^{t_2}$ belong to $\alpha Q_{i_1}$ and $\alpha Q_{i_2}$, respectively. Actually, Lemma \ref{t1t2alpha} can be used to show that there are either all possible edges between $Q_{i_1}$ and $Q_{i_2}$ or no such edges.
%Also, Lemma \ref{t1t2alpha} can be used to show that each of the sets $\alpha Q_0, \alpha Q_1, \ldots, \alpha Q_{m-1}$ is either a clique or an independent set.
Similar to Theorem~\ref{StructureOfQ}, we have the following result on the structure of the subgraph induced by $\alpha Q$.

\begin{theorem}\label{StructureOfalphaQ}
Given an odd prime power $q$ and an integer $m \ge 2$ such that $m$ divides $q+1$, the following statements hold for the subgraph of $\mathrm{GP}(q^2,m)$ induced by $\alpha Q$.\\
{\rm (1)} If $m$ is odd, then $\alpha Q_0$ is a clique, and $\alpha Q_1,\ldots, \alpha Q_{m-1}$ are independent sets; moreover, for any distinct $i_1, i_2$ such that $0 \le i_1 < i_2 \le m-1$, there are all possible edges between the sets $\alpha Q_{i_1}$ and $\alpha Q_{i_2}$ if $i_1+i_2 \equiv 0\pmod m,$ and there are no such edges if $i_1+i_2 \not\equiv 0\pmod m.$
In particular, $\alpha Q_1 \cup \alpha Q_{m-1}, \ldots, \alpha Q_{\frac{m-1}{2}} \cup \alpha Q_{\frac{m+1}{2}}$ induce $\frac{m-1}{2}$ complete bipartite graphs.\\
{\rm (2)} If $m$ is even, then $\alpha Q_0, \alpha Q_{\frac{m}{2}}$ are cliques, and $\alpha Q_1,\ldots, \alpha Q_{\frac{m}{2}-1},\alpha Q_{\frac{m}{2}+1},\ldots,\alpha Q_{m-1}$ are independent sets; moreover, for any distinct $i_1, i_2$ such that $0 \le i_1 < i_2 \le m-1$, there are all possible edges between the sets $\alpha Q_{i_1}$ and $\alpha Q_{i_2}$ if $i_1+i_2 \equiv 0\pmod m$ and there are no such edges if $i_1+i_2 \not\equiv 0\pmod m.$
In particular, if $m>2$, then $\alpha Q_1 \cup \alpha Q_{m-1}, \ldots, \alpha Q_{\frac{m}{2}-1} \cup \alpha Q_{\frac{m}{2}+1}$ induce $(\frac{m}{2}-1)$ complete bipartite graphs.
\end{theorem}
\begin{proof}
The proof is similar to the proof of Theorem~\ref{StructureOfQ}. The only difference is that we apply Lemma~\ref{t1t2alpha} instead of Lemma~\ref{t1t2}. $\square$
\end{proof}

\begin{proposition}[$\alpha Q_0$-construction]\label{alphaQ0Clique}
Let $q$ be an odd prime power and $m\ge 2$ be an integer such that $m$ divides $q+1$.
If $m$ does not divide $\frac{q+1}{2}$, then 
$\alpha Q_0$ is a clique of size $\frac{q+1}{m}$ in $\mathrm{GP}(q^2,m)$;
if $m$ divides $\frac{q+1}{2}$, then 
$\{0\} \cup \alpha Q_0$ is a clique of size $\frac{q+1+m}{m}$ in $\mathrm{GP}(q^2,m)$.
\end{proposition}
\begin{proof}
By Theorem \ref{StructureOfalphaQ}, $\alpha Q_0$ is a clique. By Lemma \ref{NeighboursOf0InQ}, the elements from $Q_0$ belong to the neighbourhood of $0$. If $m$ does not divide $\frac{q+1}{2}$, then, in view of Lemma \ref{calpha}, the element $\alpha$ is not an $m$-th power in $\mathbb{F}_{q^2}^*$, which means that the elements from $\alpha Q_0$ do not belong to the neighbourhood of $0$. On the other hand, if $m$ divides $\frac{q+1}{2}$, then $\alpha Q_0$ lies in the neighbourhood of $0$, and $\{0\} \cup \alpha Q_0$ is a clique. $\square$
\end{proof}

\medskip

Similar to Lemma~\ref{Paleymaximal}, it is shown in \cite{GKSV18} and \cite{GMS22} that both $Q_0$-construction and $\alpha Q_0$-construction produce maximal cliques in Paley graphs of square order. Again, it is tempting to extend this result to generalised Paley graphs under minor extra assumptions.

\begin{conjecture}\label{Conj2}
Assume that $p \nmid (m-1)$ and $2 \leq m \leq \frac{q+1}{3}$. Then the cliques from Proposition~\ref{0Q0Clique} and Proposition~\ref{alphaQ0Clique} are maximal.
\end{conjecture}

Using SageMath, we have verified Conjecture~\ref{Conj2} for all pairs $(q,m)$ for which $q \leq 500$ and $m \mid (q+1)$. Similar to Conjecture~\ref{Conj1}, the two assumptions in the statement of Conjecture~\ref{Conj2} cannot be dropped for the same reason. While Example~\ref{counterexample} does not give an explicit family of counterexamples when the first condition is dropped, small instances such as $\mathrm{GP}(27^2,7)$ (see also the last paragraph of Example~\ref{counterexample}), $\mathrm{GP}(125^2,21)$, and $\mathrm{GP}(243^2, 61)$ do serve the purpose. 

\subsection{Correspondences between cliques in $\mathrm{GP}(q^2,m)$}

In the previous two subsections, we constructed a few families of cliques in generalised Paley graphs. Inspired by the recent work \cite{GMS22}, where two correspondences between maximal cliques (independent sets) in Paley graphs of square order were established, in this subsection, we establish generalised correspondences for cliques in $\mathrm{GP}(q^2,m)$, where $m$ divides $q+1$.

Let us define 
the first mapping $\varphi:\mathbb{F}_{q^2} \to \mathbb{F}_{q^2}$
by the rule
$$
\varphi(\gamma):=
\begin{cases}
\frac{\gamma+1}{\gamma-1} & \text{if~~}  \gamma \not= 1,\\
~~1 & \text{if~~} \gamma=1;
\end{cases}.
$$

The correspondences are based on the following fact.

\begin{lemma}[{\cite[Proposition 7]{GMS22}}]\label{QImage} 
Let $\gamma$ be an element from $Q \setminus \{1\}$, where $\gamma = x + y\alpha$ for some $x,y \in \mathbb{F}_{q}$. Then the following formula holds
    \begin{center}
        $\displaystyle \varphi(\gamma) = \frac{y}{x-1}\alpha$.
    \end{center}
\end{lemma}

\subsubsection{$\varphi$-correspondence between cliques}

\begin{theorem}\label{phicor}
Let $q$ be an odd prime power and an integer $m>1$ such that $m$ divides $\frac{q+1}{2}$. Then, for the clique $\{0\} \cup  Q_0$ from Proposition \ref{0Q0Clique}, the image 
$\varphi(Q_0 \cup \{0\})$ coincides with the clique of size $\frac{q+1+m}{m}$ from Proposition \ref{alphaFq1}.
\end{theorem}
\begin{proof}
Note that $\varphi(1) = 1$, $\varphi(-1) = 0$ and $\varphi(0) = -1$. 
The vertices from $\varphi(Q_0\setminus\{1\})$ form a clique since the line $\{c\alpha~|~c\in\mathbb{F}_q\}$ is $m$-ary in view of Lemma \ref{calpha}. Let us prove that 1 is adjacent to any vertex from $\varphi(Q_0\setminus\{1,-1\})$. If $\gamma$ is an arbitrary element from $Q\setminus\{1,-1\}$, then $\gamma^m$ represents an arbitrary element from $Q_0\setminus\{1,-1\}$. Let us apply Lemma \ref{hom} and Lemma \ref{gammaminus1}:
$$
(\varphi(\gamma^m)-1)^{q-1} = \bigg(\frac{\gamma^m+1}{\gamma^m-1}-1\bigg)^{q-1} =
\bigg(\frac{2}{\gamma^m-1}\bigg)^{q-1} =
\frac{1}{(\gamma^m-1)^{q-1}} =
-\gamma^m.
$$
Since $-1 = \omega^{\frac{q+1}{2}}$ and $m$ divides $\frac{q+1}{2}$, the element $(-\gamma^m)$ is an $m$-th power in $Q$. Thus, 1 is adjacent to every element from $\varphi(Q_0\setminus\{1\})$.

To complete the proof, we need to show that $-1$ has the same neighbours in $\{c\alpha~|~c\in\mathbb{F}_q\}$ as $1$ does. It follows from the fact that the mapping $\varepsilon\mapsto-\varepsilon^q$ is an automorphism of $\mathrm{GP}(q^2,m)$ that stabilises the set $\{c\alpha~|~c\in\mathbb{F}_q\}$ pointwise and swaps $1$ and $-1$. $\square$ 
\end{proof}

\subsubsection{$\psi$-correspondence between cliques}
Define the second mapping $\psi:\mathbb{F}_{q^2} \to \mathbb{F}_{q^2}$ as
$$
\psi(\gamma):= \alpha \varphi(\alpha^{-1} \gamma) =
\begin{cases}
\frac{\alpha\gamma+d}{\gamma-\alpha} & \text{if~~} \gamma \not= \alpha,\\
~~~\alpha & \text{if~~} \gamma=\alpha.
\end{cases}
$$

\begin{lemma}[{\cite[Proposition 8(2)]{GMS22}}]\label{alphaQImage} 
Let $\gamma$ be an element from $Q \setminus \{1\}$, where $\gamma = x + y\alpha$ for some $x,y \in \mathbb{F}_{q}$.
    Then
    $$
    \displaystyle \psi(\alpha\gamma) = \frac{yd}{x-1}.
    $$
\end{lemma}

\begin{theorem}\label{psicor}
Let $q$ be an odd prime power, $m$ divides $q+1$. If $m$ does not divide $\frac{q+1}{2}$, then for the clique $Q_0$ from Proposition \ref{alphaQ0Clique}, the image 
$\psi(Q_0)$ coincides with the clique of size $\frac{q+1}{m}$ from Proposition \ref{Fqalpha};
if $m$ divides $\frac{q+1}{2}$, then for the clique $\{0\} \cup Q_0$ from Proposition \ref{alphaQ0Clique}, the image
$\psi(\{0\} \cup Q_0)$ coincides with the clique of size $\frac{q+1+m}{m}$ from Proposition~\ref{Fqalpha}.
\end{theorem}
\begin{proof}
Note that $\psi(\alpha) = \alpha$, $\psi(-\alpha) = 0$ and $\psi(0) = -\alpha$. 
The vertices from $\psi(\alpha Q_0\setminus\{\alpha\})$ form a clique since the line $\mathbb{F}_q$ is $m$-ary. Let us prove that $\alpha$ is adjacent to any vertex from $\psi(\alpha Q_0\setminus\{\alpha,-\alpha\})$. If $\gamma$ is an arbitrary element from $Q\setminus\{1,-1\}$, then $\alpha \gamma^m$ represents an arbitrary element from $\alpha Q_0\setminus\{\alpha,-\alpha\}$. We have
$$
\psi(\alpha \gamma^m) =
\frac{\alpha(\gamma^m+1)}{\gamma^m-1}.
$$
Then
$$
\bigg(\frac{(\alpha\gamma^m+1)}{\gamma^m-1} - \alpha\bigg)^{q-1}
=
\bigg(\frac{2\alpha}{\gamma^m-1}\bigg)^{q-1}.
$$
Note that $\beta^{\frac{q+1}{2}} = c\alpha$ for some $c \in \mathbb{F}_q^*$ and thus $(c\alpha)^{q-1} = -1$. Applying Lemma \ref{gammaminus1}, we obtain
$$
\bigg(\frac{2\alpha}{\gamma^m-1}\bigg)^{q-1} = \gamma^m,
$$
which is an $m$-th power in $Q$.
Thus, $\alpha$ is adjacent to every element from $\varphi(\alpha Q_0\setminus\{\alpha\})$. $\square$ 
\end{proof}

\medskip

Theorem \ref{phicor} and Theorem \ref{psicor} do not imply the equivalence of Conjecture~\ref{Conj1} and Conjecture~\ref{Conj2} since the correspondences are not automorphisms of the graph. Moreover, Conjecture~\ref{Conj1} and Conjecture~\ref{Conj2} are essentially different since the cliques have different geometric structure. Note that for cliques from Proposition~\ref{Fqalpha} and Proposition~\ref{alphaFq1} (Conjecture~\ref{Conj1}), almost all the vertices lie on a line (except one or two vertices). In comparison, for the cliques from Proposition~\ref{0Q0Clique} and Proposition~\ref{alphaQ0Clique} (Conjecture~\ref{Conj2}), there are almost no collinear triples of vertices due to the oval structure.

\section{Partial progress towards Conjecture~\ref{Conj1}}\label{partialresult}
Throughout this section, we work on the generalised Paley graph $\mathrm{GP}(q^2,m)$, where $m \mid (q+1)$ and $2 \leq m \leq \frac{q+1}{3}$.

For each $u \in \F_{q^2}$, let $N(u)$ denote the $\F_q$-neighbourhood of $u$, that is, 
$$
N(u):=N_{q,m}(u)=\{x \in \F_q: u-x \text{ is an } m\text{-th power in } \F_{q^2} \}.
$$
Note that if $u \in \F_{q^2} \setminus \F_q$, then the size of $N(u)$ is independent of the choice of $u$, more precisely,
\begin{equation}\label{Nu}
|N(u)|=\frac{q+1}{m}-1.
\end{equation}
This can be seen geometrically by Lemma~\ref{lem: numberoflines}. Alternatively, since $\mathrm{GP}(q^2,m)$ is strongly regular and the Delsarte-Hoffman bound (see for example \cite[Lemma 15]{AGLY22}) holds for the maximum clique $\F_q$, it follows that $\F_q$ forms a regular clique, that is, $|N(u)|$ is a constant. We would like to show that in the generic case, $N(u) \cup \{u\}$ is nearly a maximal clique, thus establishing a weaker version of Conjecture~\ref{Conj1}.

\subsection{The {\rm (}$\mathbb{F}_q,\alpha{\rm )}$-construction in general gives a nearly maximal clique}

One may expect that $N(u)$ behaves like a random subset of $\F_q$. If $N(u) \cup \{u\}$ is not maximal, then one can expand the clique by adding $v \in \F_{q^2} \setminus \F_q$. Then it is necessary that $N(u) \subset N(v)$ and thus $u,v$ share the same $\F_q$-neighbourhood, that is, $N(u)=N(v)$, since $|N(u)|=|N(v)|$ by equation~\eqref{Nu}. If this is the case, then it is natural to expect that the neighbourhood $N(u)$ must have a special algebraic structure. Recall that we are working on a Cayley graph, and we identify elements in $\F_{q^2}$ with points in $AG(2,q)$ by identifying $x+y\alpha$ with $(x,y)$; in particular, a line which is parallel to the line $\F_q$ is horizontal. The following proposition confirms the above heuristic using a geometric idea. 

\begin{proposition}\label{prop: samenhbd}
Assume that $m \mid (q+1)$ and $2 \leq m \leq \frac{q+1}{3}$. Let $u,v$ be distinct elements in $\F_{q^2} \setminus \F_q$ such that $N(u)=N(v)$ and $u-v$ is an $m$-th power in $\F_{q^2}$. Then exactly one of following two statements holds:

\begin{itemize}
    \item If $u,v$ lie on a horizontal line (that is, $u-v \in \F_q$), then $N(u)$ is a union of additive $(u-v)\F_p$-cosets of $\F_q$; in particular, $p \mid (m-1)$.
    \item If $u,v$ does not lie on a horizontal line, then there are $a \in \F_q$ and $t \in \F_q^*$, such that $t \neq 1$ and $u-a=t(v-a)$. Moreover, $N(u-a)\setminus \{0\}$ is a union of $H$-cosets of $\F_q^*$, where $H=\langle t \rangle$ is the multiplicative subgroup of $\F_q^*$ generated by $t$. In particular, if $|H|=d$, then $d \mid \gcd(q-1, \frac{q+1}{m}-2)$.
    \end{itemize}
\end{proposition}
\begin{proof}
First we assume that $u,v$ lie on a horizontal line. It follows that $u-v \in \F_q^*$. Let $x \in N(u)$.  Then $u-x=v-(x-(v-u))$ is an $m$-th power in $\F_{q^2}$, and thus $x-(v-u) \in N(v)=N(u)$. Repeating the same argument, we must have $x-2(v-u), x-3(v-u), \cdots \in N(v)$. Therefore, for each $x \in \F_q$, we have $x \in N(u)$ if and only if $x+(v-u)\F_p \subset N(u)$. We conclude that $N(u)$ must be a union of additive $(u-v)\F_p$-cosets of $\F_q$. In particular, $|N(u)|$ is a multiple of $p$, that is, $p \mid (\frac{q+1}{m}-1)$. Note that $m \mid (q+1)$ implies that $p \nmid m$, and thus $p \mid (\frac{q+1}{m}-1)$ is equivalent to $p \mid (q+1-m)$, that is, $p \mid (m-1)$.

Next we consider the case that $u,v$ do not lie on a horizontal line. Since $u-v$ is an $m$-th power in $\F_{q^2}$, the $m$-ary line passing through $u$ and $v$ intersects a point in the line $\F_q$, say $a$. Let $u'=u-a$ and $v'=v-a$, then $u', v', 0$ are collinear; say $t=v'/u' \in \F_q^*$. 

Note that $N(u')=N(u)-a=N(v)-a=N(v')$. Let $x \in N(u') \setminus \{0\}$. Then geometrically or algebraically, we see that $tx \in N(v')=N(u')$.  It follows that $t^jx \in N(u')$ for any positive integer $j$. Let $H$ be the multiplicative subgroup generated by $t$, with $|H|=d$. Then $H$ is a subgroup of $\F_q^*$, and the above argument shows that the $H$-coset containing $x$ is contained in $N(u')$. Thus, $N(u') \setminus \{0\}$ must be the union of $H$-cosets. In particular, $|N(u)|=|N(u')|=\frac{q+1}{m}-1 \equiv 1 \pmod d$, that is, $d \mid (\frac{q+1}{m}-2)$. Since $H$ is a subgroup of $\F_q^*$, we also have $d \mid (q-1)$. We conclude that $d \mid \gcd(q-1,\frac{q+1}{m}-2)$.  $\square$
\end{proof}

\medskip

The following lemma is implicit in Proposition~\ref{Fqalpha}; we include a short proof for the sake of completeness.

\begin{lemma}\label{ubar}
Let $u \in \F_{q^2} \setminus \F_q$. Then
$N(u)=N(u^q)$.
\end{lemma}
\begin{proof}
Note that for each $x\in \F_q$, we have
$u^q-x=u^q-x^q=(u-x)^q$. Since $m \mid (q+1)$, it follows that $\gcd(m,q)=1$. Thus, for each $x\in \F_q$, $u-x$ is an $m$-th power if and only if $u^q-x$ is an $m$-th power. In other words, $N(u)=N(u^q)$. $\square$
\end{proof}

\medskip

Next we use Proposition~\ref{prop: samenhbd} to deduce that the clique $N(u) \cup \{u\}$ is nearly maximal. We refer to Proposition~\ref{prop: sufficient} for a discussion on weakening the assumption $p \nmid (m-1)$ in the following theorem.

\begin{theorem}\label{thm: nearlymaximal}
Assume that $m \mid (q+1)$ and $2 \leq m \leq \frac{q+1}{3}$. Further assume that $p \nmid (m-1)$. Let $u \subset \F_{q^2}\setminus \F_q$. Let $C'$ be a maximal clique in $\mathrm{GP}(q^2,m)$ such that $N(u) \cup \{u\} \subset C'$. Then $|C'| \leq \frac{q+1}{m}-1+\gcd(q-1,\frac{q+1}{m}-2)$, which is at most $\frac{q+1}{m}+\sqrt{2(q+1)}-3$. Moreover, there is $a \in \F_q$, such that $C'$ is contained in the union of the line $\F_q$ and the line $(a+(u-a)\F_q)$; to be precise, $C' \subset N(u) \cup (a+(u-a)H)$, where $H$ is the (unique) subgroup of $\F_q^*$ with order $\gcd(q-1,\frac{q+1}{m}-2)$.
\end{theorem}

\begin{proof}
Let $C=N(u) \cup \{u\}$. Then by equation~\eqref{Nu}, $|C|=\frac{q+1}{m}$. Assume that $C$ is not maximal, yet $C'=C \sqcup \{v_1, \ldots, v_k\}$ forms a maximal clique in $\mathrm{GP}(q^2,m)$. It is clear that $\{v_1, \ldots, v_k\} \subset \F_{q^2} \setminus \F_q$. Thus, it is necessary that $N(u)=N(v_i)$ and that $u-v_i$ is an $m$-th power in $\F_{q^2}$ for each $1 \leq i \leq k$. 

Let $d=\gcd(q-1,\frac{q+1}{m}-2)$. Since $p \nmid (m-1)$, by Proposition~\ref{prop: samenhbd}, for each $1 \leq i \leq k$, there are $a_i \in \F_q$ and $t_i \in \F_q^*$ such that $t_i \neq 1, t_i^d=1$, and $u-a_i=t_i(v_i-a_i)$. 

Suppose that there are $1 \leq i<j \leq k$, such that $t_i=t_j$. Then we have $u-a_i=t_i(v_i-a_i)$ and $u-a_j=t_i(v_j-a_j)$. Since $v_i \neq v_j$, it follows that $a_i \neq a_j$. It follows that $v_i-v_j=(1-t_i)(a_j-a_i)/t_i \in \F_q$. In particular, $v_i, v_j$ lie on a horizontal line and $N(v_i)=N(v_j)$. By Proposition~\ref{prop: samenhbd},  $p \mid (m-1)$, violating our assumption. 

We conclude that $t_i \neq t_j$ for $1 \leq i<j \leq k$. Note that $t_1, t_2, \ldots, t_k$ are non-trivial $d$-th roots of unity, so $k \leq d-1$. So $C'$ has size $\frac{q+1}{m}+k \leq \frac{q+1}{m}-1+d$. Note that 
\begin{align*}
d&=\gcd\bigg(q-1,\frac{q+1}{m}-2\bigg)=\gcd\bigg(\frac{q+1}{m}-2, q-1-(q+1-2m)\bigg)\\
&=\gcd\bigg(\frac{q+1}{m}-2, 2(m-1)\bigg)
\leq \min\bigg\{\frac{q+1}{m}-2, 2(m-1) \bigg\}\\
&=\min\bigg\{\frac{q+1}{m}, 2m \bigg\}-2
\leq \sqrt{2(q+1)}-2,
\end{align*}
it follows that $|C'| \leq \frac{q+1}{m}+\sqrt{2(q+1)}-3$.

Suppose $a_i \neq a_j$ for some $1 \leq i<j \leq k$. Then by Proposition~\ref{prop: samenhbd}, $N(u-a_i)=t_iN(u-a_i)$ and $N(u-a_j)=t_jN(u-a_j)$. Let $u'=u-a_i$ and let $\Delta=a_i-a_j$. Then we have $N(u')=t_iN(u')$ and $N(u'+\Delta)=t_j N(u'+\Delta)$. Let $y \in N(u')$; we have $y/t_i \in N(u')$ since $N(u')=t_iN(u')$, and thus $y/t_i+\Delta \in N(u'+\Delta)$. It follows that $x:=(y/t_i+\Delta)/t_j-\Delta \in N(u')$ since $N(u'+\Delta)=t_j N(u'+\Delta)$. Note that we have $y=t_it_jx+t_i(t_j-1)\Delta$. Now $x \in N(u')$ implies that $t_ix \in N(u')$ and $t_ix+\Delta \in N(u'+\Delta)$. It follows that $$t_j(t_ix+\Delta)-\Delta=t_it_jx+(t_j-1)\Delta \in N(u').
$$
Thus, we have $y-(t_i-1)(t_j-1)\Delta \in N(u')$. Since $\Delta \neq 0$, $t_i \neq 1$, and $t_j \neq 1$, we see that the difference $(t_i-1)(t_j-1)\Delta \neq 0$. Thus we can repeat the arguments used in the proof of Proposition~\ref{prop: samenhbd} to conclude that $N(u')$ is a union of additive $\big((t_i-1)(t_j-1)\Delta\big) \F_p$-cosets of $\F_q$. In particular, $p \mid (m-1)$, violating our assumption.

We conclude that $a_1=a_2=\ldots=a_k$, that is, $u-a_1=t_i(v_i-a_1)$ for $1 \leq i \leq k$. Therefore, $u, v_1, \ldots, v_k$ and $a_1$ are collinear. This shows that the maximal clique $C'$ is contained in $N(u) \cup (a_1+(u-a_1)H)$, where $H$ is the (unique) subgroup of $\F_q^*$ of order $d$. In particular, $C'$ is contained in the union of the lines $\F_q$ and $a_1+(u-a_1)\F_q$. $\square$
\end{proof}

\medskip

As a quick application, Theorem~\ref{thm: nearlymaximal} implies the following corollary, which proves Conjecture~\ref{Conj1} under extra assumptions and recovers Lemma~\ref{Paleymaximal}.

\begin{corollary}\label{cor: recover}
Assume that $m \mid (q+1)$ and $2 \leq m \leq \frac{q+1}{3}$. Further assume that $p \nmid (m-1)$. Let $u \subset \F_{q^2}\setminus \F_q$. If $\gcd(q-1,\frac{q+1}{m}-2)=1$, then $N(u) \cup \{u\}$ forms a maximal clique in $\mathrm{GP}(q^2,m)$; if $\gcd(q-1,\frac{q+1}{m}-2)=2$, then $N(u) \cup \{u,u^q\}$ is the unique maximal clique in $\mathrm{GP}(q^2,m)$ that contains $N(u) \cup \{u\}$. In particular, when $m=2$, this recovers Lemma~\ref{Paleymaximal}.
\end{corollary}

\begin{proof}
If $\gcd(q-1,\frac{q+1}{m}-2)=1$, then by Theorem~\ref{thm: nearlymaximal}, $N(u) \cup \{u\}$ forms a maximal clique in $\mathrm{GP}(q^2,m)$. 

Next we consider the case $\gcd(q-1,\frac{q+1}{m}-2)=2$. Let $C=N(u) \cup \{u\}$ and let $C'$ be a maximal clique containing $C$. Theorem~\ref{thm: nearlymaximal} implies that $|C'| \leq |C|+1$ and $C' \subset N(u) \cup \{u, 2a-u\}$ for some $a \in \F_q$. 

Observe that $u \neq u^q$ since $u \notin \F_q$, and that $N(u)=N(u^q)$ by Lemma~\ref{ubar}. Moreover, if $u=x+y\alpha$, where $x,y \in \F_q$, then $u^q=x-y\alpha$ since $\alpha^q=-\alpha$. So $u-u^q=2y\alpha$ is an $m$-th power in $\F_{q^2}$ by Lemma~\ref{calpha}, and thus $N(u) \cup \{u,u^q\}$ forms a maximal clique. 

We have shown that $C$ is not maximal; it follows that $C'=C \cup \{v\}$ is a maximal clique, where $v=2a-u$ and $a \in \F_q$. Suppose that $v \neq u^q$; then $u+v=2a \in \F_q$. However, note that $u+u^q=2x \in \F_q$, and thus $u^q-v \in \F_q$ with $N(v)=N(u^q)$. By Proposition~\ref{prop: samenhbd},  $p \mid (m-1)$, violating our assumption. Therefore, $N(u) \cup \{u,u^q\}$ is the unique maximal clique containing $N(u) \cup \{u\}$. 

Finally we consider the case $m=2$. Note that 
$$\gcd\bigg(q-1, \frac{q+1}{2}-2\bigg)=\gcd\bigg(q-1,\frac{q-3}{2}\bigg)=
\begin{cases}
1 & \text{if } q \equiv 1 \pmod 4\\
2 & \text{if } q \equiv 3 \pmod 4\\
\end{cases}.
$$
Thus, Lemma~\ref{Paleymaximal} follows immediately from the above discussion. $\square$
\end{proof}

\medskip

Theorem~\ref{thm: nearlymaximal} also implies the following corollary.

\begin{corollary}\label{alpha-alpha}
Assume that $m \mid \frac{q+1}{2}$ and $2 \leq m \leq \frac{q+1}{3}$. Further assume that $p \nmid (m-1)$. Let $C'$ be a maximal clique in $\mathrm{GP}(q^2,m)$ that contains $N(\alpha) \cup \{\alpha, -\alpha\}$. Then $|C'| \leq \frac{q+1}{m}+\gcd(q-1,\frac{q+1}{m}-2)$; moreover, $C' \subset N(\alpha) \cup \alpha H$, where $H$ is the (unique) subgroup of $\F_q^*$ with order $\gcd(q-1,\frac{q+1}{m}-2)$.
\end{corollary}

\begin{proof}
Note that $\alpha^q=-\alpha$. The conclusion follows Theorem~\ref{thm: nearlymaximal} by observing that the line passing through $\alpha$ and $-\alpha$ is simply $\alpha \F_q$. $\square$
\end{proof}

\medskip

We can apply the above corollary to cubic Paley graphs. 

\begin{corollary}
Consider the cubic Paley graph $\mathrm{GP}(q^2,3)$, where $q \geq 11$ and $q \equiv 2 \pmod 3$. If $q \equiv 3 \pmod 4$, then $N(\alpha) \cup \{\alpha, -\alpha\}$ is a maximal clique; if $q \equiv 1 \pmod 4$, then $N(\alpha) \cup \{\alpha, -\alpha\}$ forms a maximal clique unless $N(\alpha)=tN(\alpha)$, in which case $N(\alpha) \cup \{\alpha, -\alpha, t\alpha, -t\alpha\}$ forms a maximal clique, where $t$ is a primitive fourth root of unity in $\F_{q}^*$.
\end{corollary}
\begin{proof}
If $q \equiv 3 \pmod 4$, then $\gcd(q-1,\frac{q-5}{3})=2$ and the conclusion follows by Corollary~\ref{alpha-alpha}. Next assume that $q \equiv 1 \pmod 4$, or equivalently $\gcd(q-1,\frac{q-5}{3})=4$. Then Corollary~\ref{alpha-alpha} states that either $N(\alpha) \cup \{\alpha, -\alpha\}$ or $N(\alpha) \cup \{\alpha, -\alpha, t\alpha, -t\alpha\}$ forms a maximal clique. Note that in the latter case, we have $N(\alpha)=N(t\alpha)=tN(\alpha)$ since $t \in \F_q^*$. Conversely, if $N(\alpha)=tN(\alpha)$, then we know that $N(\alpha)=N(t\alpha)=N(t^2\alpha)=N(t^3\alpha)$, so $\alpha,-\alpha,t\alpha, -t \alpha$ share their $\F_q$-neighborhood and thus $N(\alpha) \cup \{\alpha, -\alpha, t\alpha, -t\alpha\}$ forms a maximal clique. $\square$
\end{proof}

\begin{remark}\rm
Recall that Conjecture~\ref{Conj1} states that $N(\alpha) \cup \{\alpha, -\alpha\}$ is a maximal clique even when $q \equiv 1 \pmod 4$. Indeed, computations via SageMath shows that for all prime powers $q \leq 1200$ such that $q \equiv 5 \pmod {12}$, we have $N(\alpha)\neq tN(\alpha)$, where $t$ is a primitive fourth root of unity in $\F_{q}^*$.
\end{remark}

\subsection{Further discussions on Theorem~\ref{thm: nearlymaximal}}

Note that in Theorem~\ref{thm: nearlymaximal}, we assumed that $p \nmid (m-1)$. If we fix $m$ and $p>m-1$ (equivalently, if the characteristic of $\F_{q^2}$ is at least $m$), then the assumption $p \nmid (m-1)$ is satisfied. However, if $m$ grows as a function of the characteristic $p$, then Theorem~\ref{thm: nearlymaximal} may not be useful for $\mathrm{GP}(q^2,m)$. However, in view of Proposition~\ref{prop: samenhbd} and the proof of Theorem~\ref{thm: nearlymaximal}, in the statement of Theorem~\ref{thm: nearlymaximal}, the condition $p \nmid (m-1)$ can be replaced by the following weaker condition: $N(u)$ is not the union of additive $a\F_p$-cosets of $\F_q$ for any $a \in \F_q^*$. While this weaker condition might be difficult to check in general, the following proposition provides concrete sufficient conditions which are easy to verify. Roughly speaking, the proposition states that the condition $p \nmid (m-1)$ in the statement of Theorem~\ref{thm: nearlymaximal} can be relaxed to $m \leq q^{2/3}$ (note that when $m>q^{2/3}$, we can use Proposition~\ref{prop:upperbound} to control the size of maximal cliques; see the related discussion on Section~\ref{stability}).

\begin{proposition}\label{prop: sufficient}
Let $n=2^st$, where $t$ is an odd integer. Let $q=p^n$ and let $u \in \F_{q^2}\setminus \F_q$. Further assume that $p>(2n-1)^2$ and that one of the following conditions holds:
\begin{enumerate} [(i)]
    \item $t=1$;
    \item $m \leq q^{2/3} (p^{2^s}-1)/p^{2^s}$.
\end{enumerate}
Then $N(u)$ is not the union of additive $a\F_p$-cosets of $\F_q$ for any $a \in \F_q^*$.
\end{proposition}

\begin{proof}
We proceed with proof by contradiction. Suppose that $N(u)$ is the union of additive $a\F_p$-cosets of $\F_q$ for some $a \in \F_q^*$. Since $u$ is an $m$-th power in $\F_{q^2}$, we have $0 \in N(u)$ and thus $a\F_p \in N(u)$. This implies that every element in $u+a\F_p$ is an $m$-th power in $\F_{q^2}$, or equivalently, every element in $u/a+\F_p$ is an $m$-th power in $\F_{q^2}$ (since each element in $\F_q^*$ is an $m$-th power). Since $u \in \F_{q^2} \setminus \F_q$ and $a \in \F_q$, we have $u/a \in \F_{q^2} \setminus \F_q$.

(i) Since $t=1$, $n$ is a power of $2$. It follows that all proper subfields of $\F_{q^2}$ are contained in the subfield $\F_q$. Therefore, $u/a$ is of degree $2n$ over $\F_p$. By Corollary~\ref{cor: consecutivempowers}, there is an element in $u/a+\F_p$ which is not an $m$-th power in $\F_{q^2}$, a contradiction.

(ii) Let $\F_{p^d}$ be a proper subfield of $\F_{q^2}$ such that $(u/a+\F_q) \cap \F_{p^d} \neq \emptyset$. Then $d \mid 2n$ and $d \nmid n$ since $(u/a+\F_q) \cap \F_q=\emptyset$. In particular, $d=2^{s+1}k$ where $k$ is a proper divisor of $t$.  Since $t$ is odd, it follows that $k \leq t/3$. Since $(u/a+\F_q) \cap \F_{p^d} \neq \emptyset$,  there is $b \in \F_q$, such that $x=u/a+b \in \F_{p^d}$. Then $u/a+\F_q=(u/a+b)+\F_q=x+\F_q$ and $\F_{p^d}=x+\F_{p^d}$. It follows that $(u/a+\F_q) \cap \F_{p^d}=x+(\F_q \cap \F_{p^d})$ and thus
$$
|(u/a+\F_q) \cap \F_{p^d}|=|\F_q \cap \F_{p^d}|=p^{\gcd(n,d)}=p^{\gcd(2^st, 2^{s+1}k)}=p^{2^sk}=p^{d/2}. 
$$
Therefore, 
\begin{align*}
&\bigg|(u/a+\F_q) \bigcap \bigg(\bigcup_{d\mid 2n, d<2n} \F_{p^d}\bigg) \bigg|
=\bigg|(u/a+\F_q) \bigcap \bigg(\bigcup_{d\mid 2n, d \nmid n}  \F_{p^d}\bigg) \bigg|\\
&\leq \sum_{d\mid 2n, d \nmid n} p^{d/2} \leq \sum_{k=1}^{t/3} p^{2^sk}= \frac{p^{2^s(t/3+1)}-p^{2^s}}{p^{2^s}-1}
=\frac{q^{1/3}p^{2^s}-p^{2^s}}{p^{2^s}-1}<q^{1/3} \frac{p^{2^r}}{p^{2^r}-1}-1.
\end{align*}
Since $m \leq q^{2/3} (p^{2^s}-1)/p^{2^s}$, by equation~\eqref{Nu}, we have
$$
|N(u)|=\frac{q+1}{m}-1>\frac{q}{m}-1 \geq q^{1/3} \frac{p^{2^r}}{p^{2^r}-1}-1>\bigg|(u/a+\F_q) \bigcap \bigg(\bigcup_{d\mid 2n, d<2n} \F_{p^d}\bigg) \bigg|.
$$
Thus, there is $b \in N(u)$, such that $u/a+b/a \notin \bigcup_{d\mid 2n, d<2n} \F_{p^d}$, that is, $u/a+b/a$ is not in any proper subfield of $\F_{q^2}$. Thus, $u/a+b/a$ is of degree $2n$ over $\F_p$. However, since $N(u)$ is the union of additive $a\F_p$-cosets of $\F_q$, we must have $b+a\F_p \subset N(u)$. Thus $u-(b+a\F_p)=(u-b)+a\F_p$ consists of $m$-th powers in $\F_{q^2}$, and so does $(u-b)/a+\F_p$, which contradicts Corollary~\ref{cor: consecutivempowers}. $\square$
\end{proof}

\medskip

Next we show the assumption $p \nmid (m-1)$ in Theorem~\ref{thm: nearlymaximal} is necessary by describing an infinite family of counterexamples. Note that they also show that Proposition~\ref{prop:upperbound} is sharp and the condition (ii) in Proposition~\ref{prop: sufficient} is nearly sharp. These counterexamples are based on subfield obstructions, which often occur as a major barrier in problems arising from the structure of finite fields. 

\begin{example}\rm \label{counterexample}
Let $r$ be an odd prime power and consider the generalised Paley graph $\mathrm{GP}(q^2,m)$, where $q=r^3$ and $m=r^2-r+1$.  We have $m \mid (q+1)$ as $q+1=r^3+1=(r+1)(r^2-r+1)$. Also, note that
$$
\frac{r^6-1}{r^2-1}=r^4+r^2+1=(r^2+1)^2-r^2=(r^2-r+1)(r^2+r+1).
$$
Therefore, $m \mid \frac{r^6-1}{r^2-1}$, and thus the subfield $\F_{r^2}$ forms a clique in $\mathrm{GP}(r^6, m)=\mathrm{GP}(q^2,m)$. Note that $\omega(\mathrm{GP}(q^2,m))=q= r^3$, so \cite[Corollary 3.1]{Y22} implies that the subfield $\F_{r^2}$ in fact forms a maximal clique (for otherwise $\omega(\mathrm{GP}(r^6,m))\geq r^4$). In particular, for each $u \in \F_{r^2} \setminus \F_{r^3}=\F_{r^2} \setminus \F_r$, we have $\F_{r^2} \cap \F_{r^3}=\F_r \subset N(u)$. However, note that 
$$|N(u)|=\frac{q+1}{m}-1=\frac{r^3+1}{r^2-r+1}-1=r.$$
This forces $N(u)=\F_r$ and thus $C=\{u\} \cup N(u)$ (with size $r+1$) is not a maximal clique. Moreover, $C$ is contained in the maximal clique $\F_{r^2}$, which has size $r^2$. 

Note that $\F_{r^2}$ is a maximal clique with size $r^2=(\frac{q+1}{m}-1)^2$; which shows that Proposition~\ref{prop:upperbound} is best possible. Note that $\gcd(q-1, \frac{q+1}{m}-2)=\gcd(r^3-1, r-1)=r-1<r^2-r$. This shows that in Theorem~\ref{thm: nearlymaximal}, the assumption $p \nmid (m-1)$ is necessary. Also note that in this case, $N(u)$ is the union of additive $\F_p$-cosets of $\F_q$, and $m=r^2-r+1 \approx q^{2/3}$. This shows that the condition (ii) in Proposition~\ref{prop: sufficient} is nearly sharp. In fact, when $q=p^3$, we have $m=p^2-p+1$; and condition (ii) in Proposition~\ref{prop: sufficient} states that in this case if $m \leq p^2-p$ and $p>25$, then $N(u)$ is the union of additive $a\F_p$-cosets of $\F_q$ for any $a \in \F_q^*$.

In the case $r=3$, the corresponding graph is $\mathrm{GP}(27^2, 7)$. Computation using SageMath in fact indicated that a maximal clique in $\mathrm{GP}(27^2, 7)$ has size either $9$ or $27$. In particular, this means that there is no maximal clique of size $5=\frac{27+1}{7}+1$. $\square$
\end{example}

To end the section, we remark that most arguments used in this section also apply to a general Peisert-type graph \cite[Definition 2]{AGLY22}, which is a Cayley graph defined on $\F_{q^2}^+$, with the connection set being the union of $\F_q^*$-cosets in $\F_{q^2}^*$. In fact, most arguments we used only relied on the fact that the set of $m$-th powers in $\F_{q^2}^*$ is a union of $\F_q^*$-cosets in $\F_{q^2}^*$, or equivalently, $\mathrm{GP}(q^2,m)$ is a Peisert-type graph when $m \mid (q+1)$ (see \cite[Lemma 2.10]{AY21}). However, we decided not to include such generalisation and complication as the parameters involved in a general Peisert-type graph are not always as nice as that of generalised Paley graphs to deal with.

\section{Stability of canonical cliques} \label{stability}

The goal of this section is to improve Lemma~\ref{embed}. We follow the notations used in the previous sections.

We have seen in Example~\ref{counterexample} that Proposition~\ref{prop:upperbound} is best possible. The following theorem improves Proposition~\ref{prop:upperbound} under extra assumptions. Note that in the proof of Proposition~\ref{prop:upperbound}, we only used the fact that the generalised Paley graph $\mathrm{GP}(q^2,m)$ can be realised as the block graph of an orthogonal array (Lemma~\ref{PaleyAsOAGraph}). We are going to use the extra geometric properties of generalised Paley graphs developed in the previous sections. 

\begin{theorem}\label{thm: stable}
Let $m\mid (q+1)$ and $2\leq m \leq \frac{q+1}{3}$. If $C'$ is a maximal clique in the generalised Paley graph $\mathrm{GP}(q^2,m)$ which is not maximum, then $|C'|\leq (\frac{q+1}{m}-1)^2$; moreover, if $p \nmid (m-1)$ and $m \leq \frac{q+1}{4}$, then $|C'| \leq 1+\frac{q+1}{m}(\frac{q+1}{m}-3)$.
\end{theorem}
\begin{proof}
Let $w=(q+1)/m$. By our assumption, $w \geq 3$.  We have shown that $\mathrm{GP}(q^2,m)$ is isomorphic to the block graph $X_O$, where $O=OA(w,q)$ is the orthogonal array constructed in Section~\ref{construction}. Let $C$ be the corresponding clique of $C'$ in the block graph $X_{O}$. The bound $|C'|=|C| \leq (w-1)^2$ has been proved in Proposition~\ref{prop:upperbound}. Next we further assume that $p \nmid (m-1)$.

Without loss of generality, we may assume that $0 \in C'$. It is easy to verify the point $0 \in C'$ corresponds to the column with all zeros in the orthogonal array $O$ constructed in Section~\ref{construction}. We follow the notations used in the proof of Lemma~\ref{OAmaximal}. Note that equation~\eqref{ubmaximal} gives a more precise bound that $|C| \leq 1+w|D_w|$ with $|D_w| \leq w-2$. 

Assume that $|D_w|=w-2$ and $|C|>1+|D_w|$. Then that means there is an $i$, such that there are $w-1$ columns in $C$ with the zero entry in the $i$-th row. In particular, this implies that there is $S \subset C$ such that $|S|=w-1$, $S$ contains the column with all zeros, and $S$ is contained in a canonical clique of $X_{O}$. Moreover, there is a column $u \in C$ such that $u \notin S$. Therefore, the corresponding clique $S'$ has size $w-1$ with $0 \in S'$ and is contained in a canonical clique of $\mathrm{GP}(q^2,m)$; moreover, the corresponding vertex $u'$ is not in that canonical clique and $S \cup \{u'\}$ is a clique. Let $x \in S'$ with $0 \notin x$; then $T':=x^{-1}S'$ is a clique such that $0,1 \in T'$ and $T'$ is contained in a canonical clique of $\mathrm{GP}(q^2,m)$. By Lemma~\ref{TSziklai}, we must have $T' \subset \F_q$. Note that $T'$ has size $w-1=\frac{q+1}{m}-1$. Note that $v:=x^{-1}u' \notin \F_q$ and $T' \cup \{v\}$ is a clique in $\mathrm{GP}(q^2,m)$. Therefore, we see that $T=N(v)$ as $T \subset N(v)$ and $|N(v)|=\frac{q+1}{m}-1$ as in equation~\eqref{Nu}. Recall that $x^{-1}C'$ is a maximal clique in $\mathrm{GP}(q^2,m)$ that contains $N(v) \cup \{v\}$. Therefore, Theorem~\ref{thm: nearlymaximal} implies that $|C'| \leq \frac{q+1}{m}-1+\gcd(q-1,\frac{q+1}{m}-2) \leq w-1+w-2=2w-3$.

The above argument shows that if $|D_w|=w-2$, then $|C'| \leq 2w-3$. On the other hand, if $|D_w|\leq w-3$, then we have $|C'|=|C| \leq 1+w|D_w|\leq 1+w(w-3)$. Since $w^2-5w+4=(w-1)(w-4) \geq 0$ for $w\geq 4$, we have $|C'| \leq \max\{2w-3,1+w(w-3)\}=1+w(w-3)$ when $w \geq 4$. $\square$
\end{proof}

\medskip

The following corollary improves Lemma~\ref{embed} significantly when the graph $\mathrm{GP}(q^2,m)$ has a small edge density: it shows that canonical cliques in $\mathrm{GP}(q^2,m)$ are very stable.

\begin{corollary}\label{corstable}
Assume $m > \frac{q+1}{\sqrt{q}+1}$ and $m\mid (q+1)$. If $C$ is a clique in the generalised Paley graph $\mathrm{GP}(q^2,m)$ such that $0,1 \in C$ and $|C|> (\frac{q+1}{m}-1)^2$, then $C \subset \F_q$. Moreover, if $p \nmid (m-1)$ and $m \leq \frac{q+1}{4}$, then the previous statement holds for  $|C'|>1+\frac{q+1}{m}(\frac{q+1}{m}-3)$. 
\end{corollary}
\begin{proof}
In the case $m=q+1$ or $m=\frac{q+1}{2}$, the statement is trivial; see Corollary~\ref{(q+1)/2}.

Next we assume that $m \leq \frac{q+1}{3}$. Note that $m > \frac{q+1}{\sqrt{q}+1}$ implies that $(\frac{q+1}{m}-1)^2<q$. 
Let $C$ be a clique such that $0,1 \in C$ and $|C|> (\frac{q+1}{m}-1)^2$. Then $C$ is contained in a maximal clique $C'$. If $C'$ is not a maximum clique, then Theorem~\ref{thm: stable} implies that $|C| \leq |C'| \leq (\frac{q+1}{m}-1)^2$, violating our assumption. Thus, $C'$ is a maximum clique in $\mathrm{GP}(q^2,m)$ containing $0,1$. By Lemma~\ref{TSziklai}, we must have $C'=\F_q$. This shows that $C \subset \F_q$.

If we further assume that $p \nmid (m-1)$ and $m \leq \frac{q+1}{4}$, then the same argument leads to an improved lower bound on $|C|$. $\square$
\end{proof}

\medskip
For the rest of the section, we consider maximal cliques in the
generalised Paley graphs $\mathrm{GP}(q^2, \frac{q+1}{3})$ and $\mathrm{GP}(q^2, \frac{q+1}{4})$.

\begin{proposition}
Assume $q \equiv 2 \pmod 3$. In the generalised Paley graph $\mathrm{GP}(q^2, \frac{q+1}{3})$, a maximal clique has size either $3$ or $q$. Moreover, a maximal clique with size $q$ must be a canonical (maximum) clique, and a maximal clique with size $3$ must be affinely equivalent to the maximal clique constructed in the $(\mathbb{F}_q, \alpha)$-construction.
\end{proposition}

\begin{proof}
By Lemma~\ref{TSziklai}, a maximum clique must be a canonical clique. Let $C$ be a maximal clique in $\mathrm{GP}(q^2, \frac{q+1}{3})$ which is not maximum. Similar to the proof of Theorem~\ref{thm: stable}, without loss of generality, we may assume that $0,1 \in C$ (otherwise we can apply an affine transformation on $C$, which preserves the clique structure). Since $\F_q$ is a maximum clique and $C$ is a maximal clique which is not maximum, $C$ is not contained in $\F_q$. Thus, there is $u \in \F_{q^2} \setminus \F_q$, such that $u \in C$. Recall that by equation~\eqref{Nu}, we have $|N(u)|=2$. However, since $0,1 \in C$, this forces $N(u)=\{0,1\}$.

Note that $p \nmid (\frac{q+1}{3}-1)$ since $p \geq 3$, and that $\gcd(q-1,1)=1$. Thus, Theorem~\ref{thm: nearlymaximal} implies that $N(u) \cup \{u\}=\{0,1,u\}$ is a maximal clique. This forces $C=\{0,1,u\}$. So a maximal clique which is not maximum must have size $3$, and must be affinely equivalent to the maximal clique constructed in the $(\mathbb{F}_q, \alpha)$-construction. $\square$
\end{proof}

%\textcolor{red}{Connection to Latin square graphs? Related: Section 9 in  \url{https://arxiv.org/abs/1105.0796}}

\begin{remark}\label{subgraph}\rm
In \cite[Theorem 2.1]{DM21}, a complicated formula in terms of finite field hypergeometric functions, for the number of complete subgraphs of order $4$ in a generalised Paley graph, is given. Historically, this is related to Ramsey theory and turns out to have a close connection with Paley graphs \cite{EPS81}. The above proposition leads to an elementary approach to compute the same quantity in $\mathrm{GP}(q^2,\frac{q+1}{3})$ as we have shown that the vertices of any complete subgraph of order $4$ must be contained in a canonical clique. Recall that canonical cliques correspond to lines in the affine plane $AG(2,q)$, thus the intersection of two distinct canonical cliques has size at most $1$. The number of canonical cliques is $3q$ and thus the number of complete subgraphs of order $4$ is given by 
$$
3q \cdot \binom{q}{4}.
$$
It would be interesting to find combinatorial proof for other results in \cite{DM21}. 
\end{remark}
\begin{example}\rm
Note that the Paley graph with order $25$ is $\mathrm{GP}(25,2)=\mathrm{GP}(q^2, \frac{q+1}{3})$, where $q=5$. Thus, by Remark~\ref{subgraph}, the number of complete subgraphs of order $4$ in the graph is given by $15 \cdot \binom{5}{4}=75.$
Alternatively, by writing $25=3^2+4^2$, \cite[Corollary 2.3]{DM21} gives the same answer
$$
\frac{25 \cdot 24 \cdot (16^2-64)}{2^9 \cdot 3}=75. \square
$$
\end{example}

\begin{proposition}
Assume $p>3$ and $q \equiv 3 \pmod 4$. In the generalised Paley graph $\mathrm{GP}(q^2, \frac{q+1}{4})$, a maximal clique has size either $4,5$ or $q$. Moreover, if a maximal clique has size $4$, then it must be affinely equivalent to $\{0,1,u,u+1\}$, where $u \in \F_{q^2} \setminus \F_q$, such that $u-1,u,u+1$ are $\frac{q+1}{4}$-th powers in $\F_{q^2}$, but $u-u^q+1$ is not. 
\end{proposition}
\begin{proof}
Let $C$ be a maximal clique in $\mathrm{GP}(q^2, \frac{q+1}{4})$ which is not maximum. Since $p>3$, we have $p \nmid 3$. Thus, by Theorem~\ref{thm: stable}, $|C|\leq 5$. Similar to the previous proofs, without loss of generality, we may assume that $0,1 \in C$ and there is $u \in C$ such that $u \notin \F_q$. Thus, we have $3 \leq |C| \leq 5$. It remains the show that $|C| \neq 3$, and that $|C|=4$ only if $C$ has a special structure.

If $|C|=3$, then we must have $C=\{0,1,u\}$ and $0,1 \in N(u)$. However, equation~\eqref{Nu} states that $|N(u)|=3$, so we can expand $C$ to a large clique, contradicting the assumption that $C$ is maximal. 

Finally we consider the case $|C|=4$. There are the following two cases:
\begin{itemize}
    \item $|C \cap \F_q|=3$. In this case we must have $C=\{u\} \cup N(u)$. However, since $\gcd(q-1,4-2)=2$, Theorem~\ref{thm: nearlymaximal} states that $C \cup \{u^q\}$ is maximal, a contradiction.
    \item  $|C \cap \F_q|=2$. In this case we have $C=\{0,1,u,v\}$, where $u,v \in \F_{q^2}\setminus \F_q$. 
    Let $N(u)=\{0,1,a\}$ and $N(v)=\{0,1,b\}$. Then $a \neq b$ for otherwise $\{0,1,u,v,a\}$ still forms a clique. 
    
    Assume that $u-v \notin \F_q$; then $u,v$ do not lie on a horizontal line. Since $u-v$ is a $\frac{q+1}{4}$-th power in $\F_{q^2}$, the $\frac{q+1}{4}$-ary line passing through $u$ and $v$ intersects a point in the line $\F_q$, and that the point is in $N(u) \cap N(v)=\{0,1\}$. However, note if $u,v,0$ are collinear, or $u,v,1$ are collinear, then we are back in the previous case by a proper affine tranformation on $C$. This forces $u-v \in \F_q$. Then we must have $N(v)=N(u)+(v-u)$. Since $N(u)=\{0,1,a\}$ and $N(v)=\{0,1,b\}$, we must have $v=u \pm 1$. We may assume that $v=u+1$. Then, we must have $N(u)=\{-1,0,1\}$, that is, $u-1,u,u+1$ are $\frac{q+1}{4}$-th powers in $\F_{q^2}$. Finally, recall that Theorem~\ref{thm: nearlymaximal} states that $\{0,1,u,u^q\}$ forms a clique. So we have to make sure that $u+1-u^q$ is not a $\frac{q+1}{4}$-th power in $\F_{q^2}$, for otherwise $\{0,1,u,u+1,u^q\}$ still forms a clique. $\square$
\end{itemize}
\end{proof}
As a final remark, in view of Example~\ref{counterexample}, the condition $p>3$ in the statement of the above proposition cannot be dropped.

\section*{Acknowledgments}
L. Shalaginov is supported by Russian Science Foundation according to the research project 22-21-20018. C. H. Yip is supported by a doctoral fellowship from the University of British Columbia. The authors thank Shamil Asgarli for a careful checking of the draft. C. H. Yip thanks Greg Martin and Joshua Zahl for helpful discussions.

\end{document}